\documentclass[11pt]{amsart}
 
\usepackage[margin=1.2in]{geometry}
\usepackage{subcaption}
\usepackage{hyperref}
\usepackage{graphicx}
\usepackage{mathtools}
\usepackage{amssymb}
\usepackage{bm}
\usepackage{mathrsfs}
\usepackage[skip=10pt plus1pt, indent=25pt]{parskip}
\usepackage{array}
\usepackage{dsfont}
\usepackage[foot]{amsaddr}

\newtheorem{theorem}{Theorem}[section]
\newtheorem{remark}[theorem]{Remark}

\newtheorem{lemma}[theorem]{Lemma}

\newtheorem{claim}[theorem]{Claim}

\newtheorem{example}[theorem]{Example}
\newtheorem{corollary}[theorem]{Corollary}
\newtheorem{definition}[theorem]{Definition}

\newcommand{\eps}{\epsilon}
\newcommand{\N}{\mathbb{N}}
\newcommand{\Z}{\mathbb{Z}}
\newcommand{\R}{\mathbb{R}}
\newcommand{\C}{\mathcal{C}}
\newcommand{\Ca}{\mathcal{C}_\alpha}
\newcommand{\al}{\alpha}
\newcommand{\PP}{\mathscr{P}}
\newcommand{\mc}{\mathcal}

\DeclareMathOperator{\Int}{int}
\DeclareMathOperator{\maxCC}{maxCC}
\DeclareMathOperator{\bx}{box}
\DeclareMathOperator{\proj}{Proj}
\DeclareMathOperator{\CC}{CC}

\begin{document}

\title[Tiling with Boundaries]{Tiling with Boundaries: Dense digital images have large connected components}

\author{Kyle Fridberg}
\thanks{This material is based upon work supported by the National Science Foundation Graduate
Research Fellowship under Grant No. DGE–2139899}
\address{Center for Applied Mathematics, Cornell University}
\email{kof4@cornell.edu}
\date{12 December 2025}

\begin{abstract}
If most of the pixels in an $n \times m$ digital image are the same color, must the image contain a large connected component? How densely can a given set of connected components pack in $\Z^2$ without touching? We answer these two closely related questions for both 4-connected and 8-connected components. In particular, we use structural arguments to upper bound the ``white'' pixel density of infinite images whose white (4- or 8-)connected components have size at most $k$. Explicit tilings show that these bounds are tight for at least half of all natural numbers $k$ in the 4-connected case, and for \textit{all} $k$ in the 8-connected case. We also extend these results to finite images. 

To obtain the upper bounds, we define the \textit{exterior site perimeter} of a connected component and then leverage geometric and topological properties of this set to partition images into nontrivial regions called \textit{polygonal tiles}. Each polygonal tile contains a single white connected component and satisfies a certain maximality property. We then use isoperimetric inequalities to precisely bound the area of these tiles. The solutions to these problems represent new statistics on the connected component distribution of digital images. 
\end{abstract}
\maketitle
\section{Introduction}

A digital image is a subset of the square grid ($\Z^2$) where each grid cell has a specified color. These cells are referred to as pixels. Connected components---maximal monochromatic sets of pixels satisfying a connectivity criterion---are the simplest nontrivial objects that appear in images. Consequently, the study of connected components plays a central role in image analysis and computer vision, mathematical morphology, percolation theory, path planning, combinatorics, and many other fields \cite{Grimmett, KR}. In these contexts, 4-connected components and 8-connected components appear most naturally (e.g., see Fig. \ref{fig:conncomp}). 

\begin{figure}[ht]
\centering
\includegraphics[width=0.5\textwidth]{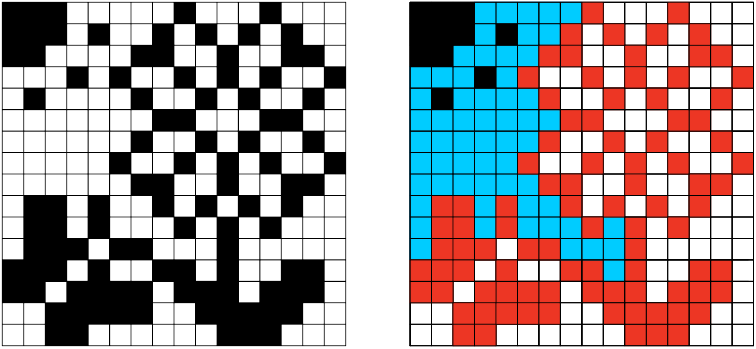}
\caption{Example of a $16 \times 16$ binary image (left), with a white 4-connected component highlighted in blue and a black 8-connected component highlighted in red (right). All $4$-connected components are $8$-connected, but the converse is not true.}
\label{fig:conncomp}
\end{figure}

We are interested in understanding how the size of the largest ``white'' component in an image is related to the white pixel density---i.e., the fraction of the image that is colored white. Specifically, we ask the following two (closely related) questions for both 4-connected and 8-connected components.

\begin{enumerate}
    \item If all white connected components in an image have size at most $k$, how large can the white pixel density be? \label{q1}
    \item If the white pixel density of an image is at least $d$, how small can the largest white connected component be? \label{q2}
\end{enumerate} 

Our main results consist of a nearly complete solution to both \eqref{q1} and \eqref{q2} for infinite images (colorings of $\Z^2$), and tight bounds on \eqref{q2} in the case where images are finite. The white pixel density of an infinite image is defined in the standard way (Definition \ref{def:rho}). Let $\al \in \{4,8\}$ be a parameter denoting either $4$- or $8$-connectivity. To describe the solution to \eqref{q1} for infinite images, let $\Phi_\al(k)$ denote the supremal white pixel density of any infinite image whose white $\al$-connected components have size at most $k$. Then we have the following result.

\begin{theorem}
\label{thm:TWBsol}
For all $k \in \N := \{1,2,3,\dots\}$,
    \begin{align*}
        \Phi_8(k) &= F_8(k) := \max_{1 \leq j \leq k} ~\frac{j}{j+\lceil 2\sqrt{j}\rceil + 1},\\
        \Phi_4(k) &\leq F_4(k) := \max_{1 \leq j \leq k} ~\frac{j}{j+\frac{\lceil 2\sqrt{2j-1}\rceil}{2}}.
    \end{align*}
    For all $k \in \mc{S}$, $\Phi_4(k) = F_4(k)$, where
    \begin{equation*}
        \mc{S} = \{1, 2, 4, 5, 7, 8, 9, 12, 13, 14, 17, 18, 19, 23, 24, 25, 26, 30,\dots\}.
    \end{equation*}
    For all $k \in \mc{U}$, $\Phi_4(k) < F_4(k)$, where
    \begin{equation*}
        \mc{U} = \{3, 6, 10, 11, 15, 16, 21, 22, 28, 29,\dots\}.
    \end{equation*}
\end{theorem}

Piecewise formulas for $F_4$ and $F_8$ are given in Appendix \ref{piecewise} (Lemmas \ref{lem:closedform} and \ref{lem:closedform8}), and the sets $\mc{S}$ and $\mc{U}$ are defined in Lemma \ref{lem:TWBsols}. We spend the majority of this paper building up the techniques to show $\Phi_\al(k) \leq F_\al(k)$. We then give explicit constructions that show the inequalities are tight. In Figure \ref{fig:densetiles}, we illustrate constructions that are tight with $\Phi_4(k)$ for small values of $k$.

\begin{figure}[ht]
\centering
\includegraphics[width=0.6\textwidth]{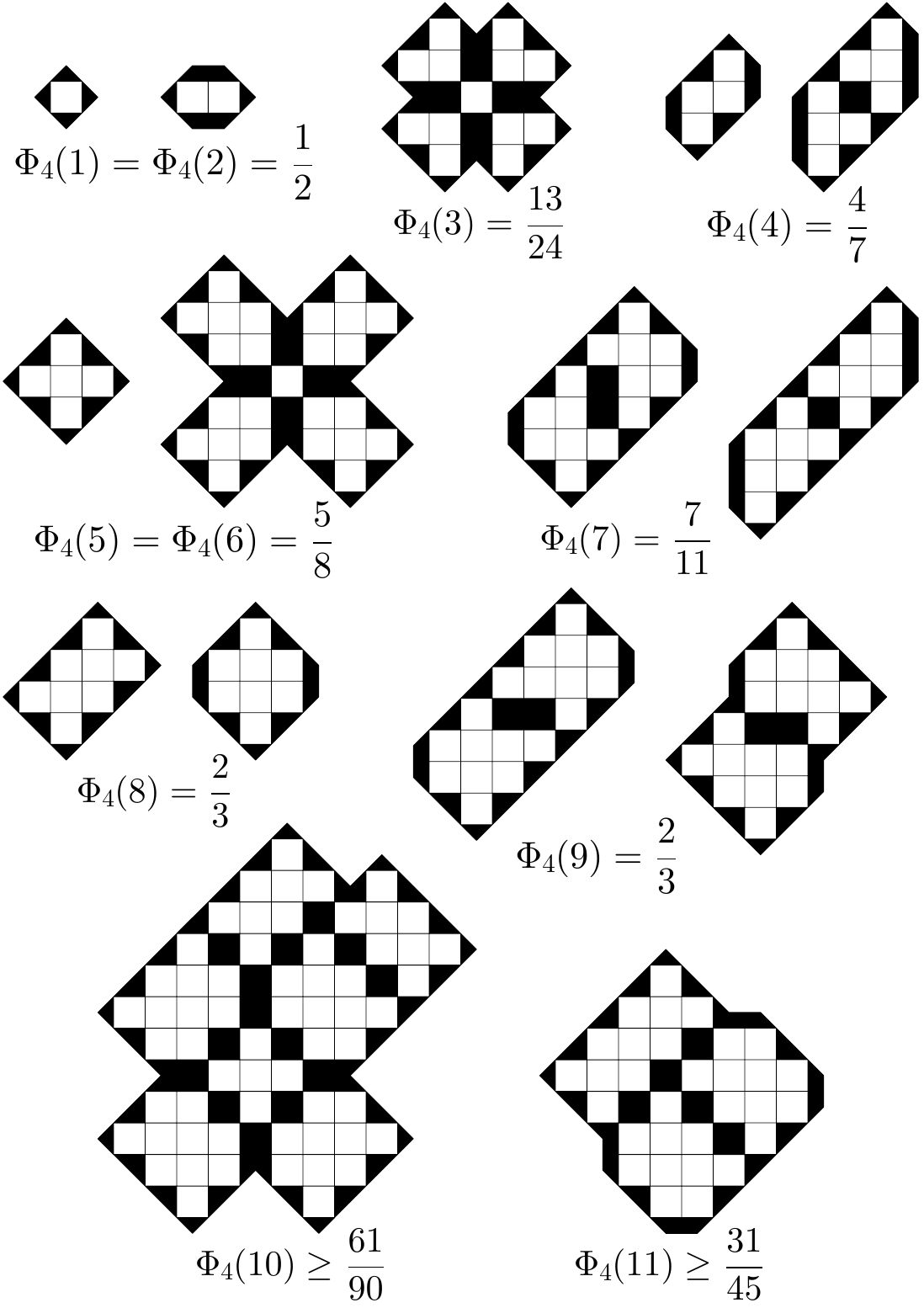}
\caption{Each of these tiles can tile the plane by translation to yield a witness to $\Phi_4(k)$ for $k \in [9]$ or a lower bound on $\Phi_4(k)$ for $k = 10,11$. The tiles that yield witnesses to $\Phi_4(1)$, $\Phi_4(3)$, and $\Phi_4(5)$ are unique.}
\label{fig:densetiles}
\end{figure} 

Note that the set $\mc{S}$ in Theorem \ref{thm:TWBsol} has natural density $\frac{1}{2}$, so we obtain the exact value of $\Phi_4(k)$ for half of all natural numbers $k$. We also compute the exact values $\Phi_4(3) = \frac{13}{24}$ and $\Phi_4(6) = \frac{5}{8}$ (Theorem \ref{thm:36}). These computations suggest that it may be difficult to exactly compute $\Phi_4(k)$ for general values of $k \in \mc{U}$. It is worth noting that the sets $\mc{U}$ and $\N \setminus \{\mc{S}\cup \mc{U}\}$ each have natural density $\frac{1}{4}$.

As a warm-up, we compute the value of $\Phi_4(1)$, and additionally deduce the value of $\Phi_4(2)$. In words, if an infinite image has white 4-connected components of size at most 1, what is the supremal white pixel density (i.e. $\Phi_4(1)$) of the image? Observe that an infinite checkerboard of alternating black and white pixels only contains white 4-connected components of size 1, and it has white pixel density $\frac{1}{2}$. Therefore $\frac{1}{2} \leq \Phi_4(1)$. On the other hand, suppose an image $I$ has white pixel density greater than $\frac{1}{2}$. An averaging argument for the density implies there must exist (infinitely many) $2 \times 2$ pixel subimages of $I$ that contain at least 3 white pixels. This implies that $I$ contains a white 4-connected component of size at least 3, hence $\Phi_4(1) = \Phi_4(2) = \frac{1}{2}$.

The simple deductive reasoning used in this warm-up does not appear to generalize, but this example highlights two important properties of $\Phi_\al(k)$, $\al \in \{4,8\}$. First, upper bounding $\Phi_\al(k)$ requires a structural argument, while lower bounding $\Phi_\al(k)$ involves explicit constructions. Second, $\Phi_\al(k)$ is not an injective function.

We now turn toward discussing question \eqref{q2}, which is essentially the inverse of question \eqref{q1}. We call the answer to \eqref{q2} the maximum guaranteed component size (MGCS), which is the minimum size of the maximal white $\al$-connected component(s) in any fixed-size image with white pixel density at least $d$. For infinite images, the MGCS is denoted by $\C_\al(d)$. In particular, any infinite image with white pixel density at least $d$ must contain a white $\al$-connected component of size at least $\C_\al(d)$, but it need not contain a white $\al$-connected component of size $\C_\al(d)+1$. By leveraging the inverse relationship between $\C_\al(d)$ and $\Phi_\al(k)$ (Lemma \ref{lem:inverses}), we obtain the following corollary of Theorem \ref{thm:TWBsol}.

\begin{corollary}
\label{cor:MGCSsol}
For all $d \in (0, 1)$,
\begin{align*}
    \mc{C}_8(d) &= F_8^{-1}(d) := \min\{k \in \N: F_8(k) \geq d\} \geq \frac{4}{(1-d)^2}-\frac{6}{1-d},\\
    \mc{C}_4(d) &\geq F_4^{-1}(d) := \min\{k \in \N: F_4(k) \geq d\} \geq \frac{2}{(1-d)^2}-\frac{4}{1-d}.
\end{align*}
Table \ref{exact} provides exact values for $\mc{C}_4(d)$ for all values of $d$ on an infinite set of half-open intervals, including $(0,\frac{61}{90}]$.
\end{corollary}

The functions $F_4$ and $F_8$ are defined in Theorem \ref{thm:TWBsol}, and piecewise formulas for $F_4^{-1}$ and $F_8^{-1}$ are given in Appendix \ref{piecewise} (Lemmas \ref{lem:F4inv} and \ref{lem:F8inv}). Figure \ref{fig:plots} illustrates the behavior of these functions.

\begin{figure}[ht]
\centering
\includegraphics[width=0.85\textwidth]{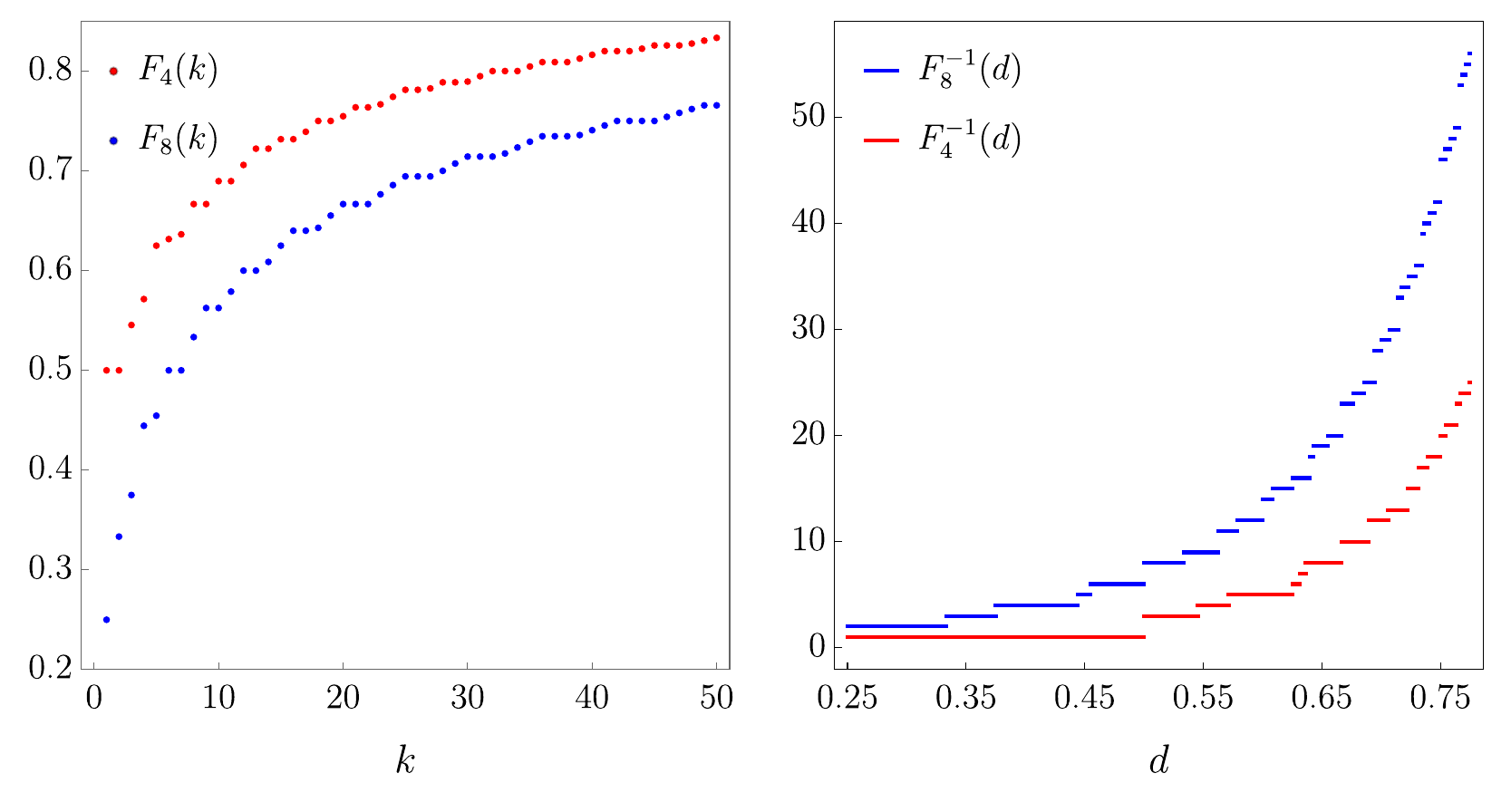}
\caption{Plot of the functions described in Theorem \ref{thm:TWBsol} (left) and Corollary \ref{cor:MGCSsol} (right).}
\label{fig:plots}
\end{figure} 

While the results we obtain for infinite images are remarkably tight, most images encountered in practice are finite. In the finite setting, we represent the maximum guaranteed component size (MGCS) by the function $\C_\al(d,n,m)$, which denotes the minimum size of the maximal white $\al$-connected component(s) in any $n \times m$ image with white pixel density at least $d$. The following tight bounds on this function represent our second main result. 

\begin{theorem}
\label{thm:bounds}
For all $d \in (0, 1]$ and $n,m \in \N$,
    \begin{align}
    \mc{C}_8(d) \geq \mc{C}_8(d,n,m) &\geq F_8^{-1}(r_8), \label{eq:C8}\\
    \mc{C}_4(d) \geq \mc{C}_4(d,n,m) &\geq F_4^{-1}(r_4), \label{eq:C4}
\end{align}
    where \begin{equation*}
        r_8(d, n, m) :=  \frac{n m d}{(n+1)(m+1)}, ~~~~ r_4(d, n, m) :=  \frac{n m d + 1}{(n+1)(m+1)}.
    \end{equation*}

\noindent For any fixed $d \in (0,1]$, there exists an $N_d \in \N$ such that all the bounds in \eqref{eq:C8} are tight for all $n,m \geq N_d$.

\noindent For any fixed $d$ that appears in the last two rows of Table \ref{exact}, there exists an $M_d \in \N$ such that all the bounds in \eqref{eq:C4} are tight for all $n,m \geq M_d$.
\end{theorem}
\begin{proof}
    Inequalities \eqref{eq:C8} and \eqref{eq:C4} are obtained by applying Lemma \ref{lem:reduction} to Corollary \ref{cor:MGCSsol}. The fact that these bounds are tight for sufficiently large $n,m$ is a direct consequence of observing that $F_8^{-1}$ and $F_4^{-1}$ are left-continuous (by Claim \ref{claim:left-cont}). 
\end{proof}

\setlength{\extrarowheight}{4.5pt}
\begin{table}[h]
\centering
\begin{tabular}{|c|c|}
\hline
\textbf{$d$} & \textbf{$\C_4(d)$} \\
\hline
$\big(0,\frac{1}{2}\big]$ & 1 \\
$\big(\frac{1}{2},\frac{13}{24}\big]$ & 3 \\
$\big(\frac{13}{24},\frac{4}{7}\big]$ & 4 \\
$\big(\frac{4}{7},\frac{5}{8}\big]$ & 5 \\
$\big(\frac{5}{8},\frac{7}{11}\big]$ & 7 \\
$\big(\frac{7}{11},\frac{2}{3}\big]$ & 8 \\
$\big(\frac{2}{3},\frac{61}{90}\big]$ & 10 \\
$\Big(\frac{4n^2-6n+2}{4n^2-2n-1},\frac{(n-1)^2+n^2}{2n^2}\Big]$ & $F_4^{-1}(d)$ \\
$\Big(\frac{4n^2-2n}{4n^2+2n-1},\frac{n}{n+1}\Big]$ & $F_4^{-1}(d)$ \\
\hline
\end{tabular}
\caption{Table of exact values of $\C_4(d)$ (for all $n \in \N_{\geq 2}$).} \label{exact}
\end{table}

\begin{figure}[ht]
    \centering
    \includegraphics[width=0.71\textwidth]{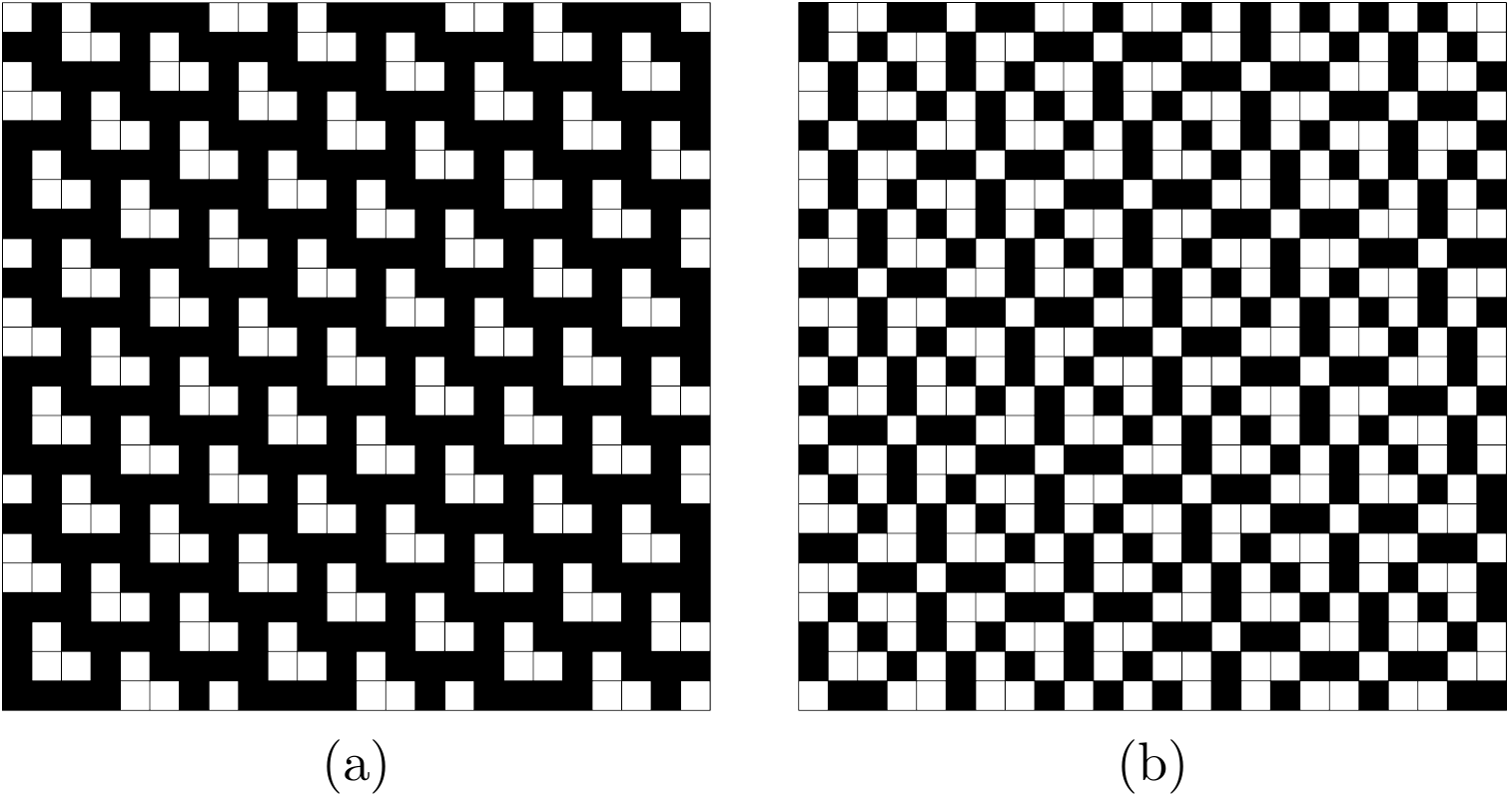}
    \caption{(a) and (b) are $24 \times 24$ pixel binary images with white pixel densities $\frac{3}{8}$ and $\frac{13}{24}$, respectively. By Theorem \ref{thm:bounds}, these parameters guarantee that (a) contains a white 8-connected component of size at least 3 and (b) contains a white 4-connected component of size at least 3. As the images show, this result is tight.}
    \label{fig:MGCSS=3}
\end{figure}

In words, Theorem \ref{thm:bounds} states that any $n \times m$ image with white pixel density at least $d$ must contain a white 4-connected component of size at least $F_4^{-1}(r_4(d,n,m))$ and a white 8-connected component of size at least $F_8^{-1}(r_8(d,n,m))$. We think of $r_4(d,n,m)$ and $r_8(d,n,m)$ as ``corrected" densities that take image size into account, and we note that $r_\al(d,n,m) \to d$  as $n,m \to \infty$. For concrete examples of images that illustrate the Theorem \ref{thm:bounds} tightness result, see Figures \ref{fig:MGCSS=3}, \ref{fig:dj-images}, and \ref{fig:5-15}.

\subsection{Related work}

To our knowledge, the questions considered here have not been studied previously, however they are related to a number of other works.

First, Bernoulli site percolation involves analyzing certain properties of the connected component distribution on bicolorings of the square grid (and other lattices). In this model, the color of each grid cell is independently assigned to be white or black by sampling a Bernoulli$(p)$ random variable, where $p$ is a fixed constant. Many properties of the resulting connected component distribution on the infinite grid have been analytically characterized \cite{Grimmett}. However, the finite setting is less well understood \cite{CDY, Richey}. Here we are interested in a deterministic (or worst-case) analogue of Bernoulli site percolation.

Second, the fixed-magnetization Ising model is a variant of the classical Ising model that is quite similar to the setup of this work. This model assigns an \textit{energy} to each function from $V \subset \Z^2$ to $\{+1, -1\}$, where the total number of $+1$ and $-1$ values in the image of $V$ are fixed. The \textit{spin configuration} $V$ is just a different notation for images with a fixed white pixel density. In the limit as the temperature parameter approaches 0 in this model, the lowest energy spin configuration minimizes the number of neighboring sites in $V$ with opposite spins, resulting in the formation of a boundary-minimizing connected component known as the Wulff shape \cite{DKS}. Boundary-minimizing components also appear in this work for a similar reason. However, in our setting these components are also size-constrained, so they form dense tilings of the grid rather than appearing in isolation.

The most closely related work is perhaps the work of van Dalen \cite{vanDalen}, which also considers the size of an image's maximum 4-connected component in an information-limited setting. Specifically, van Dalen lower bounds the size of the maximum white 4-connected component given that the sum of the boundary lengths of all white 4-connected components in an image is known. Because there is not a clear relationship between white pixel density and the sum of the connected component boundary lengths, our techniques are quite different from those used in \cite{vanDalen}.  

In terms of techniques, 4-connected component (polyomino) and 8-connected component isoperimetry results play a central role in this work \cite{A06,  BBS, HH, Sieben}. To obtain tight results, we leverage a complete description of the minimal site perimeter of components with a fixed size as well as the maximal components with a fixed site perimeter. Fortunately, these problems have been previously studied in \cite{A06} and \cite{Sieben}. Additionally, we use classical results concerning the duality of $4$-connected and $8$-connected components on $\Z^2$ \cite{Rosenfeld2, Rosenfeld}.

\subsection{Methods}

Since we are interested in monochromatic connected components, we can assume without loss of generality that all digital images are binary (black and white) for the remainder of this paper. For conciseness, we sometimes refer to $\al$-connected components, $\al \in \{4,8\}$, as ``connected components'' or just ``components.''

To leverage tools from planar geometry, we first view an image $I$ as a coloring of $\R^2$ in the natural way. We then carefully construct a partition of $I$ that depends only on the individual connected components, not on their spatial arrangement. We call these partition elements \textit{polygonal tiles}---they are precisely the disjoint, maximum-area subsets of $\R^2$ that can be uniquely assigned to each connected component without any knowledge of the other connected components in $I$. We can then upper bound the density of $I$ by the maximum density of any polygonal tile that $I$ contains. These results can be carried over to finite images via an approximate finite to infinite reduction (Lemma \ref{lem:reduction}). 

While the intuition behind this technique is clear, much of the first half of this paper will focus on making this precise. We first define the \emph{exterior site perimeter} of a white connected component, which is the minimal set of black pixels that forms a boundary around the exterior of the component. This definition is closely related to the well-studied notion of the site perimeter, though we are not aware of previous works that characterize it. The digital Jordan curve theorem (\cite{Rosenfeld}; see Theorem \ref{thm:jct}) allows us to use the exterior site perimeter to define a simple polygon in $\R^2$ that contains the entire white connected component in its interior. This is essentially the polygonal tile (e.g., see Figures \ref{fig:exampletile} and \ref{fig:exampletile2}). We then use Pick’s theorem (\cite{Rosenholtz}; see Theorem \ref{thm:Picks}) and isoperimetry results (\cite{A06, Sieben}; see Lemmas \ref{lem:minperim}, \ref{lem:minperim8}, \ref{lem:minpext}) to show that a sizable fraction of each polygonal tile must be colored black.

Representing connected components by their polygonal tiles effectively reduces the connected component packing problem on $\Z^2$ into a polygon tiling problem in $\R^2$. To show that our upper bounds are tight, we tile $\R^2$ using the polygonal tiles of two special families of connected components---one for the 4-connected case and one for the 8-connected case. These connected component families have one component of each size $k$, and this component has minimal site perimeter for its size (Fig. \ref{fig:QR}). The 8-connected case follows an elegant square spiral construction introduced by Harary and Harborth \cite{HH}, and we show an analogous diamond spiral construction in the 4-connected case. In the 8-connected case, the tiling behavior is particularly simple, and hence the maximum white pixel density is always achieved.

\subsection{Outline}

We conclude the introduction by briefly outlining the rest of the paper. In Section \ref{sec:prelims}, we record various geometric and topological results for digital images as well as providing a complete set of isoperimetric inequalities for the site perimeter of 4-connected and 8-connected components. Section \ref{sec:problemdefs} focuses on formally defining the Tiling with Boundaries (TWB) and Maximum Guaranteed Component Size (MGCS) problems and proving some basic results. We define the exterior site perimeter in Section \ref{sec:pext} and show that it is a simple cycle. This allows us to construct polygonal tiles in Section \ref{sec:PT}, leading to the desired image partition. Together, these results allow us to prove the main TWB upper bound (and MGCS lower bound) in Section \ref{sec:bounds}. In Section \ref{sec:constructions}, we use the spiral connected components to generate lattice tilings of $\R^2$ (infinite images) that prove our previous bounds are tight in most cases. In this section, we further show how to restrict these infinite images to obtain finite images where our bounds are tight. We carefully compute the values of $\Phi_4(3)$ and $\Phi_4(6)$ in Section \ref{sec:4conncomp}, and then list a few possible future research directions in Section \ref{sec:future}. 
 
\section{Preliminaries}
\label{sec:prelims}

\subsection{Binary images and connected components}

A binary image can be formally defined as a map from a subset of $\Z^2$ to the set $\{\text{black}, \text{white}\}$. We adopt the canonical coloring whenever we view a binary image as a coloring of $\R^2$---namely every unit square grid cell (centered at a point in $\Z^2$) has a uniform black or white coloring. There are a number of functions over binary images that we must formally define.

\begin{definition}
\label{def:rho}
    For a finite image $I$, the white pixel density $\rho(I)$ is the number of white pixels in $I$ divided by the total number of pixels in $I$. If $I$ is an infinite image (a coloring of $\Z^2$) and $I_{j \times j}$ is the $j \times j$ subset of $I$ that is centered at the origin, with $j$ odd, then 
\begin{equation*}
    \rho(I) := \liminf_{j\to\infty} ~\frac{\text{number of white pixels in } I_{j \times j}}{j^2}.
\end{equation*}
\end{definition}

Following the notation of Klette and Rosenfeld \cite{KR}, we define 4-connectivity and 8-connectivity in terms of two metrics: $d_4$ and $d_8$. Specifically, we define $d_4(x,y) := |x_1 - y_1| + |x_2 -y_2|$ to denote the $\ell_1$ metric and $d_8(x,y) := \max\{|x_1-y_1|,|x_2-y_2|\}$ to denote the $\ell_\infty$ metric. We say that two pixels $x$ and $y$ are 4-adjacent if $d_4(x,y) = 1$ and they are 8-adjacent of $d_8(x,y) = 1$.  Note that 4-adjacency implies 8-adjacency. Taking $\al \in \{4,8\}$, two pixels $x$ and $y$ are $\al$-connected if either $x = y$ or if they are contained in a sequence of pixels where each pixel is $\al$-adjacent to its predecessor. An $\al$-connected component is a maximal set of mutually $\al$-connected pixels that all have the same color. Importantly, we assume all connected components are finite unless otherwise stated.

Just as $\ell_1$ and $\ell_\infty$ are dual norms, $4$-connectivity and $8$-connectivity are dual notions of connectivity on $\Z^2$. For this reason---and because most of our results hold for both notions of connectivity---it is helpful to define a dual integer pair $(\al,\beta) = (4,8)$ or $(\al,\beta) = (8,4)$. Since connectivity is a fundamental property, the parameter $\al$ will become ubiquitous. If this is distracting, simply imagine $\al = 4$ and $\beta = 8$ wherever these parameters appear. 

\begin{definition}
    Let $\maxCC_\al(I)$ denote the maximum size of any white $\al$-connected component contained in image $I$. If the maximum does not exist, we set $\maxCC_\al(I) = +\infty$.
\end{definition}

\begin{definition}
    The canonical representative of a connected component $T$ is the unique translation of $T$ whose bottom-leftmost vertex is $(0,0)$. Let $\CC_\al(I)$ denote the set of canonical representatives of the white $\al$-connected components contained in an image $I$. 
\end{definition}

Importantly, we let $S_\al(k)$ denote the set of canonical representatives of all (white) $\al$-connected components of size at most $k$. We now assign a graph structure to images in order to encode the correct notion of connectivity in this context.

\begin{definition}[Graph of $I$]
\label{def:abig}
    Given an image $I$, let $I_\al$ denote the vertex $2$-colored graph with vertices in the domain of $I$ and edges between any two white vertices that are $d_\al$-adjacent and any two black vertices that are $d_{\beta}$-adjacent. 
\end{definition}

\begin{figure}[ht]
\centering
\includegraphics[width=\textwidth]{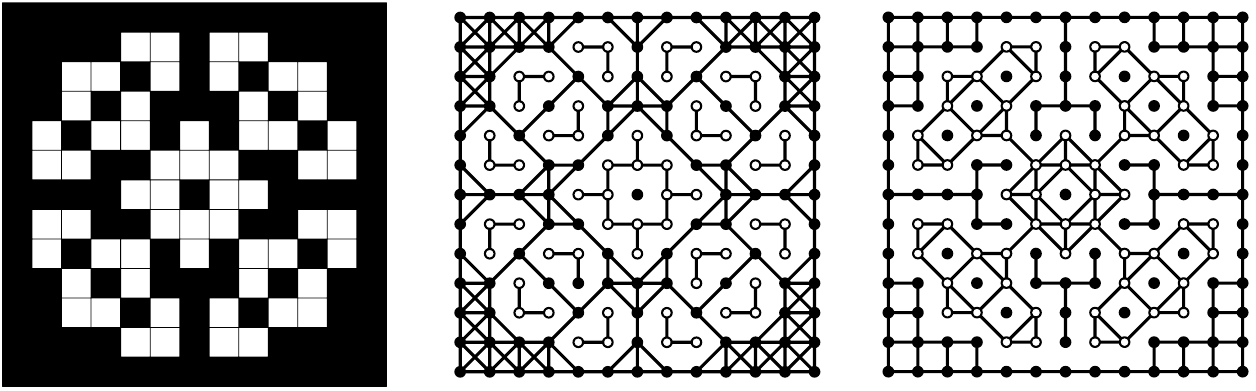}
\caption{Example of an image $I$ (left) and the corresponding graphs $I_4$ (middle) and $I_8$ (right).}
\label{fig:labeledimg}
\end{figure} 

Viewing binary images asymmetrically, as described in Definition \ref{def:abig}, enforces that white vertices appear in $\al$-connected components and black vertices appear in $\beta$-connected components. The topological utility of this asymmetric image representation is a classical result \cite{Rosenfeld}. The following theorem is the most important consequence of enforcing this notion of connectivity. 

\begin{theorem}[Digital Jordan curves \cite{Rosenfeld}]
\label{thm:jct}
    Let $S \subset \Z^2$ and $\overline{S} = \Z^2 \setminus S$. Then:
    \begin{enumerate}
        \item If $S$ is a $\beta$-connected simple cycle, then $\overline{S}$ consists of a finite and an infinite $\al$-connected component and each point in $S$ is $d_\al$-adjacent to both components.
        \item If $\overline{S}$ consists of a finite and an infinite $\al$-connected component and each point in $S$ is $d_\al$-adjacent to both components, then $S$ is a $\beta$-connected simple cycle.
    \end{enumerate}
    We call $(1)$ the digital Jordan curve theorem and $(2)$ the converse of the digital Jordan curve theorem.
\end{theorem}

We consider any subgraph $H \subset I_\al$ to inherit its connectivity from $I_\al$, so we make no notational distinction between $H$ and its vertex set. Moreover, the size of the vertex set of $H$ is given by $|H|$. 

\begin{definition}
    For any $\al$-connected component $T$, the \textit{site perimeter} $P_\al(T)$ is the set of points (in $\Z^2 \setminus T$) that are $d_\al$-adjacent to $T$.
\end{definition}

If $T$ is a white $\al$-connected component in an image graph $I_\al$, maximality of $T$ implies that $P_\al(T) \subset I_\al$ contains only black vertices. It will be helpful to define the \textit{characteristic graph} of the site perimeter to better understand the properties of this set. To obtain this graph, construct the image whose black vertex set is $P_\al(T)$, and then endow this image with the connectivity described in Definition \ref{def:abig}. The following definition provides notation for this construction. See Figures \ref{fig:exampletile} and \ref{fig:exampletile2} for examples.

\begin{definition}[Characteristic graph of the site perimeter of $T$]
\label{def:chargraph}
    Given an $\al$-connected component $T$, define the binary image $\mathds{1}_{P_\al(T)}: \Z^2 \to \{\text{black},\text{white}\}$ as follows:
    \begin{equation*}
        \mathds{1}_{P_\al(T)}(z) := \begin{cases}
            \text{black}, & \text{if } ~z \in P_\al(T),\\
            \text{white}, & \text{otherwise.}
        \end{cases}
    \end{equation*}
    We then denote the characteristic graph of $P_\al(T)$ as $G_\al(T) := (\mathds{1}_{P_\al(T)})_\al$.
\end{definition}

\subsection{Connected component isoperimetry}

If $T$ is an $\al$-connected component, we want to have a complete understanding of the relationship between $|T|$ and $|P_\al(T)|$. This will consist of answering the following two questions for both $\al = 4$ and $\al = 8$:
\begin{enumerate}
    \item If $|T|$ is fixed, what is the minimum value of $|P_\al(T)|$? \label{q1i}
    \item If $|P_\al(T)|$ is fixed, what is the maximum value of $|T|$? \label{q2i}
\end{enumerate}

One of the earliest studies of connected component isoperimetry was by Harary and Harborth \cite{HH}. They used a square spiral to generate minimal \textit{edge perimeter} $4$-connected components of any size. We call these components $\mc{R}_k$ (see Fig. \ref{fig:QR}(b)), and we observe that they are also isoperimetrically optimal in a slightly different way. In particular, among all 8-connected components $T$ of size $k$, $|P_8(T)|$ is minimized by $T = \mc{R}_k$ (Lemma \ref{lem:minperim8}). 

In the case of $4$-connectivity, Wang and Wang constructed a sequence of minimal site perimeter connected components \cite{WW}. Later works answered the questions \eqref{q1i} and \eqref{q2i} stated above \cite{A06, Sieben}. In analogy with the square spiral, we observe there is a diamond spiral $\mc{Q}_k$ (see Fig. \ref{fig:QR}(a)) that generates minimal site perimeter 4-connected components of any size. Specifically, among all 4-connected components $T$ of size $k$, $|P_4(T)|$ is minimized by $T = \mc{Q}_k$ (Lemma \ref{lem:minperim}).

\begin{figure}[ht]
\centering
\includegraphics[width=\textwidth]{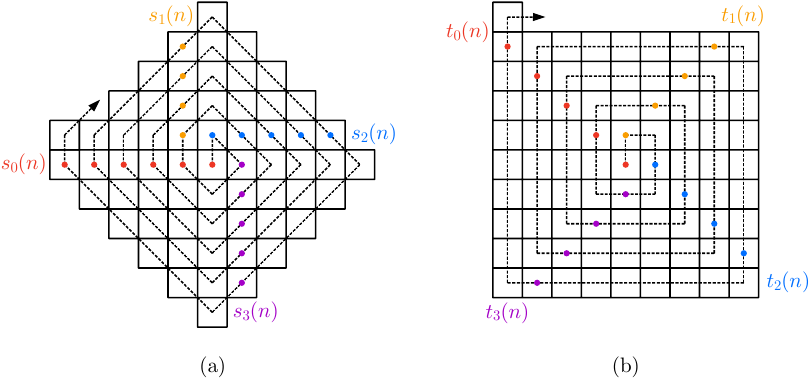}
\caption{$\mc{Q}_k$ is the union of the first $k$ integer points (or grid cells) along the dotted curve in (a). Similarly, $\mc{R}_k$ is the union of the first $k$ integer points (or grid cells) along the dotted curve in (b). The points marked $s_i(n)$ and $t_i(n)$ correspond with connected components that are optimal in the sense of question \eqref{q2i}.}
\label{fig:QR}
\end{figure} 

\begin{definition}[Sieben's component families \cite{Sieben}]
\label{def:ABD}
    For $x = (x_1, x_2) \in \R^2$ and integers $n,k$ such that $1 \leq k \leq n$, we define the following $4$-connected components via their vertex sets:
\begin{align*}
    \mc{A}_{k,n} &:= \mc{U}_{k,n} \cap \Z^2, ~~~\mc{U}_{k,n} := \{x~| -n < x_2-x_1<n ~\text{ and }-n<x_2+x_1<2k-n\}.\\
    \mc{B}_{k,n} &:= \mc{V}_{k,n} \cap \Z^2, ~~~\mc{V}_{k,n} := \{x~| -n < x_2-x_1<n+1 ~\text{ and }-n<x_2+x_1<2k-n+1\}.\\
    \mc{D}_{k,n} &:= \mc{W}_{k,n} \cap \Z^2, ~~~\mc{W}_{k,n} := \{x~| -n < x_2-x_1<n ~\text{ and }-n<x_2+x_1<2k-n+1\}.
\end{align*}
We impose the additional constraints $(k=1 \implies n = 1)$ for $\mc{A}_{k,n}$ and $2 \leq n$ for $\mc{D}_{k,n}$. 
\end{definition}
To keep our notation consistent, the components in Definition \ref{def:ABD} are a translation (and rotation, in the case of $\mc{D}_{k,n}$) of the equivalently named components given in \cite{Sieben}. These components give a complete answer to question \eqref{q2i}, which we record in the following lemma.

\begin{lemma}[Maximum size of 4-connected components \cite{A06, Sieben}]
\label{lem:maxarea}
The complete list (up to isometries) of maximum-size $4$-connected components with site perimeter $p \geq 8$ are provided by the following table, where $n := \lfloor \frac{p}{4} \rfloor$.

\setlength{\extrarowheight}{4.5pt}
\centering
\begin{tabular}{|c|c|c|}
\hline
$p \mod 4$ & Connected component(s) & Connected component size\\
\hline
$0$ & $\mc{A}_{n,n}$ & $s_0(n) := (n-1)^2+n^2$ \\
$1$ & $\mc{D}_{n,n}$ & $s_1(n) := n(2n-1)$  \\
$2$ & $\mc{A}_{n,n+1}$, $\mc{B}_{n,n}$ & $s_2(n) := 2n^2$ \\
$3$ & $\mc{D}_{n,n+1}$ & $s_3(n) := n(2n+1)$  \\
\hline
\end{tabular}
\end{lemma}

\begin{remark}
\label{rem:Q=sieben}
    Up to isometries of $\Z^2$, $\mc{A}_{n,n} = \mc{Q}_{s_0(n)}$, $\mc{D}_{n,n} = \mc{Q}_{s_1(n)}$, $\mc{B}_{n,n} = \mc{Q}_{s_2(n)}$, $\mc{D}_{n,n+1} = \mc{Q}_{s_3(n)}$.
\end{remark}

\noindent We now answer question \eqref{q1i} for 4-connected components, which is effectively the inverse of question \eqref{q2i}.  

\begin{lemma}[Minimum site perimeter of 4-connected components \cite{A06, Sieben}]
\label{lem:minperim}
  For all $4$-connected components $T$ with $|T|=k$,
  \begin{equation*}
      |P_4(T)| \geq \sigma_4(k) := \Big\lceil 2\sqrt{2k-1} \Big\rceil+2.
  \end{equation*}
  Additionally, for every $k \in \N$ there exists at least one $4$-connected component of size $k$ that attains this minimum site perimeter size.
\end{lemma}

\begin{claim}
    $|P_4(\mc{Q}_k)| = \sigma_4(k)$ for all $k \in \N$.
\end{claim}

\begin{proof}
    We use a local argument to characterize how the site perimeter changes from $\mc{Q}_k \to \mc{Q}_{k+1}$. Up to isometries of $\Z^2$, there are only four cases to consider, and we obtain
    \begin{equation}
    \label{eq:Qdiff}
        |P_4(\mc{Q}_{k+1})| - |P_4(\mc{Q}_k)|= \begin{cases}
            1, & \text{if }~ k = s_i(n)\\
            0, & \text{otherwise}
        \end{cases}
    \end{equation}
    The result follows by applying Lemma \ref{lem:maxarea} and Remark \ref{rem:Q=sieben}.
\end{proof}

\noindent In the $8$-connected component case, we have the following answer to question \eqref{q1i}.

\begin{lemma}[Minimum site perimeter of 8-connected components]
\label{lem:minperim8}
    For all $8$-connected components $T$ with $|T|=k$,
    \begin{equation*}
        |P_8(T)| \geq \sigma_8(k) :=  2\Big\lceil2\sqrt{k}\Big\rceil+4.
    \end{equation*}
    Additionally, for every $k \in \N$ there exists at least one $8$-connected component of size $k$ that attains this minimum site perimeter size.
\end{lemma}

\begin{proof}
To prove Lemma \ref{lem:minperim8}, let $R(T)$ be the bounding rectangle of an 8-connected component $T$, i.e. the smallest 4-connected rectangle that surrounds $T$. Such a rectangle has four corner cells, hence $|R(T)|$ is always four more than the edge perimeter of the rectangular region it encloses. Consequently, Claim \ref{claim1} implies $|R(T)| \geq 2\big\lceil 2 \sqrt{|T|}\big\rceil + 4$, and Lemma \ref{lem:boxbound} finishes the proof. The fact that the minimum site perimeter size is always attained follows from observing that the projection map described in Lemma \ref{lem:boxbound} is a bijection from $R(\mc{R}_k)$ to $P_8(\mc{R}_k)$. This fact can also be deduced from a simple local argument.

If $T$ has a fixed site perimeter size of at most $2n + 4$, Claim \ref{claim1} implies the size of $T$ is uniquely maximized when $T$ is a rectangular $8$-connected component with dimensions $\lceil \frac{n}{2} \rceil \times \lfloor \frac{n}{2} \rfloor$. By Lemma \ref{lem:boxbound}, the maximum size of $T$ is strictly smaller if its site perimeter is $2n + 5$. We explicitly obtain such a component by deleting a cell from the interior of the maximum-size rectangular component with site perimeter $2n + 4 \geq 14$. We record these results in slightly more detail in Lemma \ref{lem:maxarea8}.
\end{proof}

\begin{lemma}[Maximum size of 8-connected components]
\label{lem:maxarea8}
The list (up to isometries) of maximum-size $8$-connected components with site perimeter $p \geq 14$ are provided by the following table, where $n := \lfloor \frac{p}{8} \rfloor$. We use $*$ to represent an unspecified component set that grows as a function of $n$.

\setlength{\extrarowheight}{4.5pt}
\centering
\begin{tabular}{|c|c|c|}
\hline
$p \mod 8$ & Connected component(s) & Connected component size\\
\hline
$0$ & $\mc{R}_{t_0(n)}$ & $t_0(n) := (2n-1)^2$ \\
$1$ & $*$ & $t_0(n) -1$  \\
$2$ & $\mc{R}_{t_1(n)}$ & $t_1(n) := 2n(2n-1)$ \\
$3$ & $*$ & $t_1(n)-1$  \\
$4$ & $\mc{R}_{t_2(n)}$ & $t_2(n) := (2n)^2$ \\
$5$ & $*$ & $t_2(n) - 1$  \\
$6$ & $\mc{R}_{t_3(n)}$ & $t_3(n) := 2n(2n+1)$ \\
$7$ & $*$ & $t_3(n)-1$ \\
\hline
\end{tabular}
\end{lemma}

\subsection{Geometric results}

We now state two important theorems.

\begin{theorem}[Pick's theorem \cite{Rosenholtz}]
\label{thm:Picks}
    Suppose $P$ is the interior of a simple polygon whose vertices are integer points (i.e. elements of $\Z^2$). Then
    \begin{equation*}
        \text{area}(P) = i+\frac{b}{2}-1,
    \end{equation*}
    where $i$ is the number of integer points that are strictly inside $P$ and $b$ is the number of integer points on the boundary of $P$. 
\end{theorem}

\begin{theorem}[Translational monotiling \cite{GBN}]
\label{thm:exacttile}
    Let $R$ be a subset of the plane that is homeomorphic to the closed disc and whose boundary is piecewise $C^2$ with a finite number of inflection points. Then $R$ tiles the plane by translation if and only if $R$ can be completely surrounded by translated copies of itself.
\end{theorem}

\section{Tiling with Boundaries and the Maximum Guaranteed Component Size}
\label{sec:problemdefs}

In this section, we formally frame the main questions \eqref{q1} and \eqref{q2} that we posed in the introduction. We also show how these two questions are directly related to each other, which will allow us to answer both questions simultaneously. 

\begin{definition}[Tiling with Boundaries (TWB)]
\label{def:TWB}
    If $S$ is a set of (canonically represented) $\al$-connected components and $\mathcal{I}_{n \times m}$ is the set of all $n \times m$ binary images,
    \begin{equation*}
        \Phi_\al(S,n,m) := \max \{\rho(I)  :  I \in \mathcal{I}_{n \times m}, \CC_\al(I) \subseteq S \}.
    \end{equation*}
    Similarly, if $\mc{I}$ is the set of all infinite binary images,
    \begin{equation*}
        \Phi_\al(S) = \sup \{\rho(I) : I \in \mathcal{I}, \CC_\al(I) \subseteq S \}.
    \end{equation*}
     Determining the value of $\Phi_\al(S,n,m)$ or $\Phi_\al(S)$ is termed the $S$-TWB problem. 
     
     \noindent Here, we focus almost exclusively on the $S_\al(k)$-TWB problem for infinite images, so its solution is denoted $\Phi_\al(k)$ for simplicity.
\end{definition}

The name \textit{tiling with boundaries} in Definition \ref{def:TWB} is due to the fact that computing $\Phi_\al(S)$ requires constructing the densest tiling of $\Z^2$ by translating elements $T \in S$ under the constraint that every translated copy of $T$ is surrounded by a black boundary, $P_\al(T)$, in the tiling. This constraint is both necessary and sufficient to ensure that no translated elements of $S$ can merge to create a component that is not a translated element of $S$. For simplicity, we will often omit the distinction between a component $T \in S$ and the elements of its equivalence class under translation. 

We say that an $n \times m$ image $I$ is a witness to $\Phi_\al(S,n,m)$ if $\CC_\al(I) \subseteq S$ and $\rho(I) = \Phi_\al(S,n,m)$. Similarly, a witness to $\Phi_\al(S)$ is either an infinite image $I$ such that $\CC_\al(I) \subseteq S$ and $\rho(I) = \Phi_\al(S)$ or, if the supremum is not attained, the witness is a sequence of infinite images $(I_k)_{k=1}^{\infty}$ with $\CC_\al(I_k) \subseteq S$ for all $k$ and $\lim_{k \to \infty} \rho(I_k) = \Phi_\al(S)$. The supremum is attained in all of the TWB instances we solve in this paper, so all of our witnesses will be single images.

The white pixel density of an infinite binary image that only has white $\al$-connected components in $S$ is always a lower bound on $\Phi_\al(S)$. In this way, we can think of lower bounds on $\Phi_\al(S)$ as proofs of the existence of infinite images that attain a certain density, while upper bounds are a proof of the nonexistence of infinite images that attain a certain density.

Recall that we called the answer to question \eqref{q2} the maximum guaranteed component size (MGCS). We now provide a working definition. 

\begin{definition}[Maximum Guaranteed Component Size (MGCS)]
\label{def:MGCS}
If $\mathcal{I}_{n\times m}$ is the set of all $n \times m$ binary images, then
\begin{equation*}
    \Ca(d, n, m) := \min \{\maxCC_\al(I) : I \in \mathcal{I}_{n \times m}, \rho(I) \geq d\}.
\end{equation*}
Similarly, let $\mc{I}$ be the (uncountable) set of all infinite binary images. Then we set
\begin{equation*}
    \Ca(d) := \min\{\maxCC_\al(I) : I \in \mathcal{I}, \rho(I) \geq d\}.
\end{equation*}
Determining the value of $\Ca(d,n,m)$ or $\Ca(d)$ is termed the MGCS problem.
\end{definition}

Note that the codomain of $\maxCC_\al$ is $\N\cup\{+\infty\}$, hence $\Ca(d)$ is a minimum over a set of positive integers (and $+\infty$), which is well-defined. The definitions of $\Ca(d, n, m)$ and $\Ca(d)$ imply that there always exists at least one image $I$ with density $\rho(I) \geq d$ whose maximal white $\al$-connected component(s) have the same size as the MGCS. One other important property that follows immediately from Definition \ref{def:MGCS} is that $\Ca(d, n, m)$ and $\Ca(d)$ are nondecreasing functions of $d$. 

The TWB problem is more general than the MGCS problem because it aims to maximize the density of images that contain arbitrary component sets. However, this difference purely reflects a notational choice. In this paper, we mainly focus on calculating $\Phi_\al(k)$, the solution to the $S_\al(k)$-TWB problem. This component set is special for the following reason. If we know $\Phi_\al(k) = d$, then any tiling of the plane that achieves density $d' > d$ must contain at least one white $\al$-connected component of size at least $k+1$, hence $\Ca(d') \geq k+1$. We formalize this observation in the following lemma.

\begin{lemma}
\label{lem:inverses}
    $\Ca(d)$ and $\Phi_\al(k)$ are inverses:
    \begin{enumerate}
        \item $\displaystyle \Ca(d) \geq \min \{k \in \N:~\Phi_\al(k) \geq d\}$,
        \item $\displaystyle \Phi_\al(k) = \sup \{d \in (0,1]:~ \Ca(d) \leq k\}$.
    \end{enumerate}
    \noindent Moreover, equality is attained in $(1)$ if there exists an image that is a witness to $\Phi_\al(k)$.
\end{lemma}

\begin{proof}
Let $k$ be the minimum value where $\Phi_\al(k) \geq d$, so $d^* := \Phi_\al(k-1) < d$. By definition, any infinite image that achieves density greater than $d^*$ must then contain an $\al$-connected component of size at least $k$, so $\Ca(d) \geq k$. Note that if there exists a single image that is a witness to $\Phi_\al(k)$, then the image must have density at least $d$, so $\Ca(d) = k$. 

On the other hand, let $d$ be the unique value that satisfies $\Ca(d- \eps) \leq k$ and $\Ca(d+ \eps) > k$ for all $\eps > 0$. Then there cannot exist an infinite image with density $d+\eps$ that contains only white $\al$-connected components in $S_\al(k)$, so $\Phi_\al(k) \leq d$. However, the condition $\Ca(d- \eps) \leq k$ for all $\eps > 0$ implies that there must exist a sequence of infinite images, each containing only white $\al$-connected components in $S_\al(k)$, whose density approaches $d$ from below. Hence $\Phi_\al(k) \geq d$.
\end{proof}

Lemma \ref{lem:inverses} shows that solving the $S_\al(k)$-TWB problem gives a lower bound on the solution to the MGCS problem for infinite images. This is a valuable result because the behavior of $\Phi_\al(k)$ is easier to directly analyze than $\Ca(d)$. 

Dealing with finite images introduces boundary effects, which complicates the analysis of the the $S_\al(k)$-TWB and MGCS problems. Fortunately, computing the MGCS for infinite images leads to an asymptotically tight bound on the MGCS of finite images (Lemma \ref{lem:reduction}). Theorem \ref{thm:bounds} is a combination of this lemma and Corollary \ref{cor:bounds}. 

\begin{lemma}[Approximate finite to infinite reduction]
\label{lem:reduction}
   For all $d \in (0,1]$,
    \begin{align*}
        &\Ca(d) \geq \Ca(d,n,m) \geq \Ca(r_\al(d,n,m)), \text{ where }\\
        & r_8(d,n,m) := \frac{nmd}{(n+1)(m+1)}, ~~ r_4(d,n,m) := \frac{nmd+1}{(n+1)(m+1)}. 
    \end{align*}
\end{lemma}

\begin{proof} 
For any $d$, the definition of the MGCS implies there must exist an infinite image $I$ with density $d'\geq d$ that has $\maxCC_\al(I) = \Ca(d)$. Since $\rho(I) = d'$, Definition \ref{def:rho} implies that for every $\eps > 0$ there exists an $N_{\eps} \in \N$ such that for all $j \geq N_{\eps}$, the finite subimage $I_{j \times j}$ has density at least $d' -\eps$. For fixed $\eps$, take $j \geq N_\eps$ such that the least common multiple of $n$ and $m$ divides $j$, allowing us to partition $I_{j \times j}$ into $n \times m$ subimages. An averaging argument implies that there exists an $n \times m$ subimage $I' \subset I_{j \times j}$ with density at least $d' -\eps$, and hence $I'$ contains at least $\lceil nm(d'-\eps)\rceil$ white pixels. For sufficiently small $\eps$, $\lceil nm(d'-\eps)\rceil \geq nmd$, and hence $I'$ is an $n \times m$ subimage of $I$ with density at least $d$. Clearly $I'$ has no white $\al$-connected component of size greater than $\Ca(d)$, hence $\Ca(d) \geq \Ca(d,n,m)$.
 
To obtain the lower bound, for any density value $d$ there exists an $n \times m$ image $H$ with density $\rho(H) \geq d$ and $\maxCC_\al(H) = \Ca(d,n,m)$. Orient $H$ in the plane such that the two sides of length $n$ are vertical and the two sides of length $m$ are horizontal. Then expand $H$ by placing a $1 \times m$ black component along the bottom edge of $H$, and place a $n \times 1$ black component along the leftmost edge of $H$. Let $p$ be a single pixel that is white if $\al = 4$, black if $\al = 8$. After placing $p$ in the bottom left corner of the construction, we have transformed $H$ into an $(n+1) \times (m+1)$ image $I$ with density $r_\al(d,n,m)$. We can then then generate an infinite image $J$ that has the same density as $I$ by shifting copies of $I$ up, down, left, and right, as shown in Fig. \ref{fig:imgtiling}.

\begin{figure}[ht]
\centering
\includegraphics[width=0.6\textwidth]{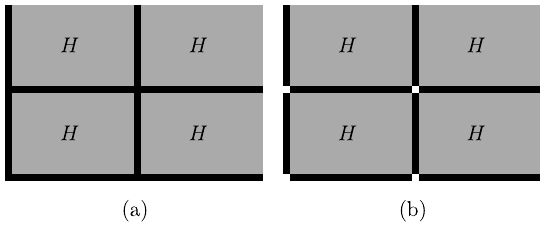}
\caption{The image $H$ can be tiled with boundaries such that white $\al$-connected components in the infinite tiling are isolated to copies of $H$. This construction for $\al = 8$ and $\al = 4$ is shown in (a) and (b), respectively.}
\label{fig:imgtiling}
\end{figure} 

Observe that the edges of every translated copy of $H$ are surrounded by a complete $P_\al$ perimeter of black pixels in $J$, hence the largest white $\al$-connected components in $J$ have size $\maxCC_\al(H)$. In the case $\al = 4$, there are isolated white pixels that appear in $J$ but not $H$, however these are irrelevant to the MGCS because $d > 0$, so we already know $H$ contains a white component of size at least 1. Therefore $\Ca(d,n,m) \geq \Ca(r_\al(d,n,m))$.
\end{proof}

\section{Exterior site perimeter}
\label{sec:pext}

We now return to considering the geometry of connected components. We know that any image containing an $\al$-connected component $T$ must also contain the site perimeter $P_\al(T)$. Since the topology of the site perimeter can be complicated, it will be helpful to partition the site perimeter into two sets according to the following definition (see examples in Figures \ref{fig:exampletile} and \ref{fig:exampletile2}).

\begin{definition}[$P_\al^{ext}, P_\al^{int}$, and $E_\al$]
    If $T$ is an $\al$-connected component, let $E_\al(T)$ denote the unique infinite $\al$-connected component in $G_\al(T)$ (recall Definition \ref{def:chargraph}). Define the exterior site perimeter $P_\al^{ext}(T)$ to be the graph induced by the set of vertices in $P_\al(T)$ that are $d_\al$-adjacent to a vertex in $E_\al(T)$. We then call $P_\al^{int}(T) := P_\al(T) \setminus P_\al^{ext}(T)$ the interior site perimeter. 
\end{definition}

By definition, switching the color of a vertex in $G_\al(T)$ from black to white connects $T$ and $E_\al(T)$ if and only if the vertex is in $P_\al^{ext}(T)$. In this work, it is often the case that $P_\al^{int}(T)$ is empty. This partition of the site perimeter is valuable because $P_\al^{ext}(T)$ is a simple cycle (Lemma \ref{lem:simple}), but $P_\al(T)$ need not even be connected. 

While Lemma \ref{lem:simple} is intuitively clear, the fact that connected components can have a nested structure adds subtlety to the proof.

\begin{lemma}
\label{lem:simple}
    $P_\al^{ext}(T)$ is a simple $\beta$-connected cycle.
\end{lemma}

\begin{proof}
We claim that for any $\al$-connected component $T$, $P_\al^{ext}(T)$ is a simple $\beta$-connected cycle. This claim can be proven by structural induction. In particular, we can show $P_\al^{ext}(T)$ is a simple $\beta$-connected cycle if $P_{ext}(\tau_i)$ is a simple $\beta$-connected cycle for each of the (maximal) white $\al$-connected components $\tau_1, \dots, \tau_n \subset G_\al(T)$ that are distinct from $T$. The intuition here is that if $P_\al^{ext}(T) = P_\al(T)$, the desired result is an immediate consequence of the converse of the digital Jordan curve theorem (Theorem \ref{thm:jct}, part (2)). If $P_\al^{ext}(T) \neq P_\al(T)$, we can switch each vertex in $P_\al^{int}(T)$ from black to white to reduce to the easy case. For purely organizational purposes, we use a tree to describe the ``nested'' component relationship. Specifically, we let $\mc{T}_T$ denote a rooted tree whose nodes are connected components, with root node $T$. The children of any node $\tau$ are precisely the (maximal) white $\al$-connected components that are distinct from $\tau$ in $G_\al(\tau)$. 

For the base case, $\mc{T}_T$ is a single node with no children, hence all of the white vertices in $G_\al(T)$ are contained in either $T$ or $E_\al(T)$. Switch each vertex $v \in P_\al^{int}(T) \subset G_\al(T)$ from black to white, and update the edge set of the resulting graph $G_\al(T)'$ accordingly. By the definition of $P_\al^{int}(T)$, each vertex in $P_\al^{int}(T)$ that is switched from black to white becomes $\al$-connected to $T$ but does \textit{not} become $\al$-connected to $E_\al(T)$, hence the number of white $\al$-connected components is unchanged by this process. Therefore $G_\al(T)'$ contains only one finite white $\al$-connected component, which we call $T'$, and one infinite white $\al$-connected component, $E_\al(T)$. Moreover, each black vertex $v \subset G_\al(T)'$ is $d_\al$-adjacent to both $T'$ and $E_\al(T)$ by the definition of $P_\al^{ext}(T)$. Part (2) of Theorem \ref{thm:jct} then states that $P_\al^{ext}(T)$ is a simple $\beta$-connected cycle. 

Now suppose that the claim is true for the $\al$-connected components $\tau_1, \dots, \tau_n$ that are the children of $T$ in $\mc{T}_T$. To prove the desired claim for $T$, we consider what happens to $G_\al(T)$ when all the vertices in $P_\al^{int}(T)$ are changed from black to white. In the first case, we assume all the exterior site perimeters of the children of $T$ have nonempty intersection with $P_\al^{int}(T)$. In the second case, we assume the exterior site perimeter of at least one child of $T$ has empty intersection with $P_\al^{int}(T)$, and hence is a subset of $P_\al^{ext}(T)$. In the proof, it becomes clear that this second case is impossible, but we do not need to show this. 

If $P_{ext}(\tau_i) \cap P_\al^{int}(T) \neq \emptyset$ for all $i \in [n]$, changing every vertex in $P_\al^{int}(T)$ from black to white must connect $T$ and $\tau_i$, and hence $P_\al^{ext}(T)$ is a simple $\beta$-connected cycle by the same argument we used in the base case. If (hypothetically) there exists an $i \in [n]$ such that $P_{ext}(\tau_i) \cap P_\al^{int}(T) = \emptyset$, then $P_{ext}(\tau_i) \subset P_\al^{ext}(T)$. Consider the subgraphs induced by the vertex sets of $P_{ext}(\tau_i)$ and $P_\al^{ext}(T) \setminus P_{ext}(\tau_i)$. By the inductive hypothesis and the assumption that $P_{ext}(\tau_i) \subset P_\al^{ext}(T)$, $P_{ext}(\tau_i)$ is a simple $\beta$-connected cycle whose vertices are $d_\al$-adjacent to both $T$ and $E_\al(T)$. It follows by the digital Jordan curve theorem that $T$ is contained in the interior face of $P_{ext}(\tau_i)$ and $E_\al(T)$ is contained in the exterior face. Furthermore, any vertex in $P_\al^{ext}(T) \setminus P_{ext}(\tau_i)$ that is in the interior face of $P_{ext}(\tau_i)$ cannot be $d_\al$-adjacent to $E_\al(T)$, and any vertex in the exterior face of $P_{ext}(\tau_i)$ cannot be $d_\al$-adjacent to $T$. Therefore $P_\al^{ext}(T) \setminus P_{ext}(\tau_i) = \emptyset$, so $P_\al^{ext}(T) = P_{ext}(\tau_i)$ is a simple $\beta$-connected cycle. 

Finally, note that the site perimeter uniquely characterizes each connected component, so $P_\al(\tau) \subsetneq P_\al(T)$ for each $\tau$ that is a child of $T$. It follows that the height of $\mc{T}_T$ is upper bounded by $|P_\al(T)|$, which is finite because $T$ is a finite component. The result follows.
\end{proof}

By Lemma \ref{lem:simple}, the embedding of $P_\al^{ext}(T)$ in $\R^2$ is a simple polygon. This allows us to define the following point sets. 

\begin{definition} If $T$ is an $\al$-connected component, let $V_\al(T)$ denote the open polygonal region in $\R^2$ whose boundary is $P_\al^{ext}(T)$. Then let $Z_\al(T) := V_\al(T) \cap \Z^2$.
\label{def:VZ}
\end{definition}

\noindent An important consequence Lemma \ref{lem:simple} is that
\begin{equation}
\label{eq:Z2partition}
    Z_\al(T) \sqcup E_\al(T) \sqcup P_\al^{ext}(T)  = \Z^2.
\end{equation}

Unsurprisingly, the interior site perimeter of isoperimetrically optimal components is always empty (Lemma \ref{lem:minpext}). We know that $\sigma_\al(k)$---the minimum value of $|P_\al(T)|$ among all $\al$-connected components $T$ of size $k$---is a nondecreasing function (Lemmas \ref{lem:minperim} and \ref{lem:minperim8}). We employ this fact to show the following result.

\begin{lemma} For all $\al$-connected components $T$ with $|T|=k$, $|P_\al^{ext}(T)| \geq \sigma_\al(k)$.
\label{lem:minpext}
\end{lemma}

\begin{proof}
Suppose $T$ is an $\al$-connected component of size $k$, and consider the (white) $\al$-connected component $Z_\al(T)$. Note that $P_\al(Z_\al(T)) = P_\al^{ext}(T)$ by construction. Additionally, $|Z_\al(T)| = k + \ell$ for some $\ell \geq 0$ because $T \subset Z_\al(T)$. Since $\sigma_\al(k)$ is a nondecreasing function, $|P_\al^{ext}(T)| = |P_\al(Z_\al(T))| \geq \sigma_\al(k+\ell) \geq \sigma_\al(k)$. 
\end{proof}

\section{Polygonal tiles}
\label{sec:PT}

We aim to cover all the white area of a binary image with disjoint regions---each containing a single white $\al$-connected component---whose densities have a nontrivial upper bound. As a warm-up, the simplest such covering is the following. To each white $\al$-connected component $T \subset \Z^2$ in an image, associate the region $\{x \in \R^2: d_\al(x,T) < 1\}$. While this covering is nearly optimal in the $8$-connected component case, it is too small to yield a tight upper bound on $\Phi_4(k)$. 

This leads to the following question: if we want a disjoint covering of each white $\al$-connected component in an image, what is the \textit{maximal} region that can be independently assigned to cover each connected component? Here, each region should depend on only the component it contains, not on any of the other components in the image. The goal of this section is to show that these maximal covering regions---called \textit{polygonal tiles}---can be constructed from $V_\al(T)$, which we defined in the previous section. We first prove an important metric property of this set.

\begin{lemma}[$V_\al(T)$ is a Voronoi region]
\label{lem:voronoi}
\begin{equation*}
    V_\al(T) = \{x \in \R^2 ~|~ d_\al(x, Z_\al(T)) < d_\al(x, E_\al(T))\}.
\end{equation*}
\end{lemma}

\begin{proof}
    To obtain equality between the two sets, we show that all points in $V_\al(T)$ satisfy the desired property, and all points not in $V_\al(T)$ do not satisfy this property. We will first use the fact that $Z_\al(T) \sqcup E_\al(T) \sqcup P_\al^{ext}(T) = \Z^2$ \eqref{eq:Z2partition} to generate a suitable partition of $\R^2$ into three regions. For two of these regions, a distance argument is sufficient to show the desired result, and for the third region---the points near the boundary of $V_\al(T)$---we proceed via a simple case analysis. 

    Any partition on $\Z^2$ induces a trivial partition on $\R^2$ by coordinate-wise rounding elements of $\R^2$ to the nearest integer. Our partition is nearly identical to this trivial partition, but coordinates in $\Z + \frac{1}{2}$ will be treated differently. For any $K \subseteq \Z^2$, define $\bx(K) = \{x \in \R^2 ~|~ d_8(x,K) \leq \frac{1}{2}\}$, i.e. $\bx(K)$ is the union of closed $1 \times 1$ square regions centered at each point in $K$. Then, letting $\Int(S)$ denote the topological interior of a set $S$,
    \begin{equation}
        \Int(\bx(Z_\al(T))) \sqcup \Int(\bx(E_\al(T))) \sqcup \bx(P_\al^{ext}(T)) = \R^2.
    \end{equation} 
    This result follows from the fact that $\Int(\bx(Z_\al(T))) \cup \Int(\bx(E_\al(T))) = \Int(\bx(Z_\al(T) \cup E_\al(T)))$ because $Z_\al(T)$ and $E_\al(T)$ are not 4-connected, and $\Int(\bx(W)) \sqcup \bx(\Z^2 \setminus W) = \R^2$ for any $W \subset \Z^2$. In particular, set $W = Z_\al(T) \cup E_\al(T)$ and apply \eqref{eq:Z2partition}.

     Since $P_\al^{ext}(T)$ is an 8-connected cycle (for $\al = 4,8$), the edges of the cycle are entirely contained in $\bx(P_\al^{ext}(T))$. It follows that $\Int(\bx(Z_\al(T))) \subseteq V_\al(T)$ and $\Int(\bx(E_\al(T))) \subseteq \R^2 \setminus V_\al(T)$. We can now verify the desired result for these two subsets.
     
     Observe that if $z \in \Z^2$ and $x \in \Int(\bx(z))$, then $d_\al(x,z) < d_\al(x, \Z^2 \setminus z)$. Consequently, any $x \in \Int(\bx(Z_\al(T)))$ is closer (in the $d_\al$ metric) to $Z_\al(T)$ than it is to any vertex in $E_\al(T)$. Similarly, any $x \in \Int(\bx(E_\al(T)))$ is closer to $E_\al(T)$ than it is to any vertex in $Z_\al(T)$.

     For the final case, $P_\al^{ext}(T)$ is a $\beta$-connected simple cycle, so there are only finitely many pixel arrangements we need to consider. In the cases $\al = 8$ and $\al = 4$, there are respectively only three and six possible contours around a vertex $b \in P_\al^{ext}(T)$ (up to isometries of $\Z^2$). In each case, we can directly compute the set $\{x \in \bx(b) ~|~ d_\al(x, Z_\al(T)) < d_\al(x, E_\al(T))\}$ (Figs. \ref{fig:localboundary8} and \ref{fig:localboundary4}). In the figures, we observe that the boundary between the points $x \in \bx(b)$ that have $d_\al(x, Z_\al(T)) < d_\al(x, E_\al(T))$ (red) and $d_\al(x, Z_\al(T)) > d_\al(x, E_\al(T))$ (blue) is precisely the embedding of the edges of the graph $P_\al^{ext}(T)$ in $\R^2$, which proves the desired result.

\begin{figure}[ht]
    \centering
    \includegraphics[width=0.47\textwidth]{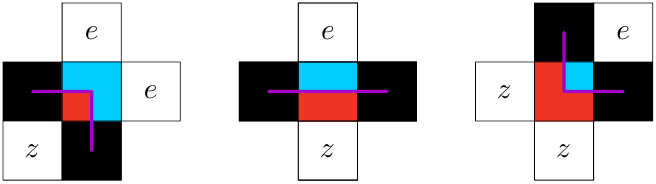}
    \caption{The three local boundary contours of $P_8^{ext}(T)$. The colored central box is $\bx(b)$. We specify the vertices $z \in Z_8(T)$, $e \in E_8(T)$, and color $x \in \bx(b)$ red if $d_8(x, Z_8(T)) < d_8(x, E_8(T))$, blue if $d_8(x, Z_8(T)) > d_8(x, E_8(T))$, and purple if $d_8(x, Z_8(T)) = d_8(x, E_8(T))$. Not all of the vertices $z \in Z_8(T)$ need to be white as pictured.}
    \label{fig:localboundary8}
\end{figure}

\begin{figure}[ht]
    \centering
    \includegraphics[width=\textwidth]{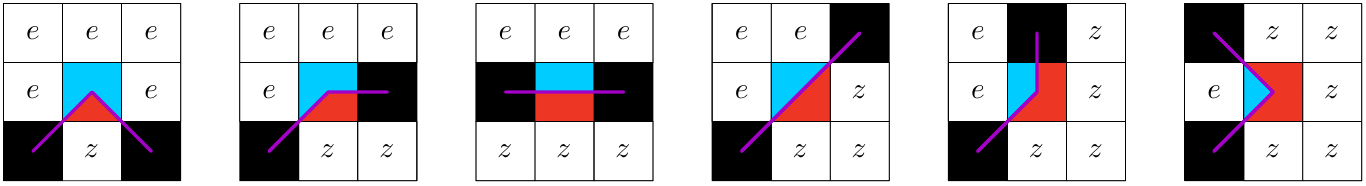}
    \caption{The six local boundary contours of $P_4^{ext}(T)$. The colored central box is $\bx(b)$. We specify the vertices $z \in Z_4(T)$, $e \in E_4(T)$, and color $x \in \bx(b)$ red if $d_4(x, Z_4(T)) < d_4(x, E_4(T))$, blue if $d_4(x, Z_4(T)) > d_4(x, E_4(T))$, and purple if $d_4(x, Z_4(T)) = d_4(x, E_4(T))$. Not all of the vertices $z \in Z_4(T)$ need to be white as pictured.}
    \label{fig:localboundary4}
\end{figure}
\end{proof}

If $T_1$ and $T_2$ are distinct white components in an image, it may be the case that $V_\al(T_1) \subset V_\al(T_2)$. Since we want to construct \textit{disjoint} sets that cover each white component, this motivates the following definition.

\begin{definition}[Polygonal tile of $T$]
\label{def:pt}
Given an $\al$-connected component $T$, let $\mc{M}$ be the set of all white $\al$-connected subgraphs of $G_\al(T)$ that are contained in $Z_\al(T) \setminus T$. Then the polygonal tile of $T$ is the set $\PP_\al(T)$ colored by $\mathds{1}_{P_\al(T)}$ in the canonical way, where
    \begin{equation*}
        \PP_\al(T) := V_\al(T) \setminus \bigcup_{\tau \in \mc{M}} V_\al(\tau).
    \end{equation*}
\end{definition}

\begin{figure}[ht]
\centering
\includegraphics[width=\textwidth]{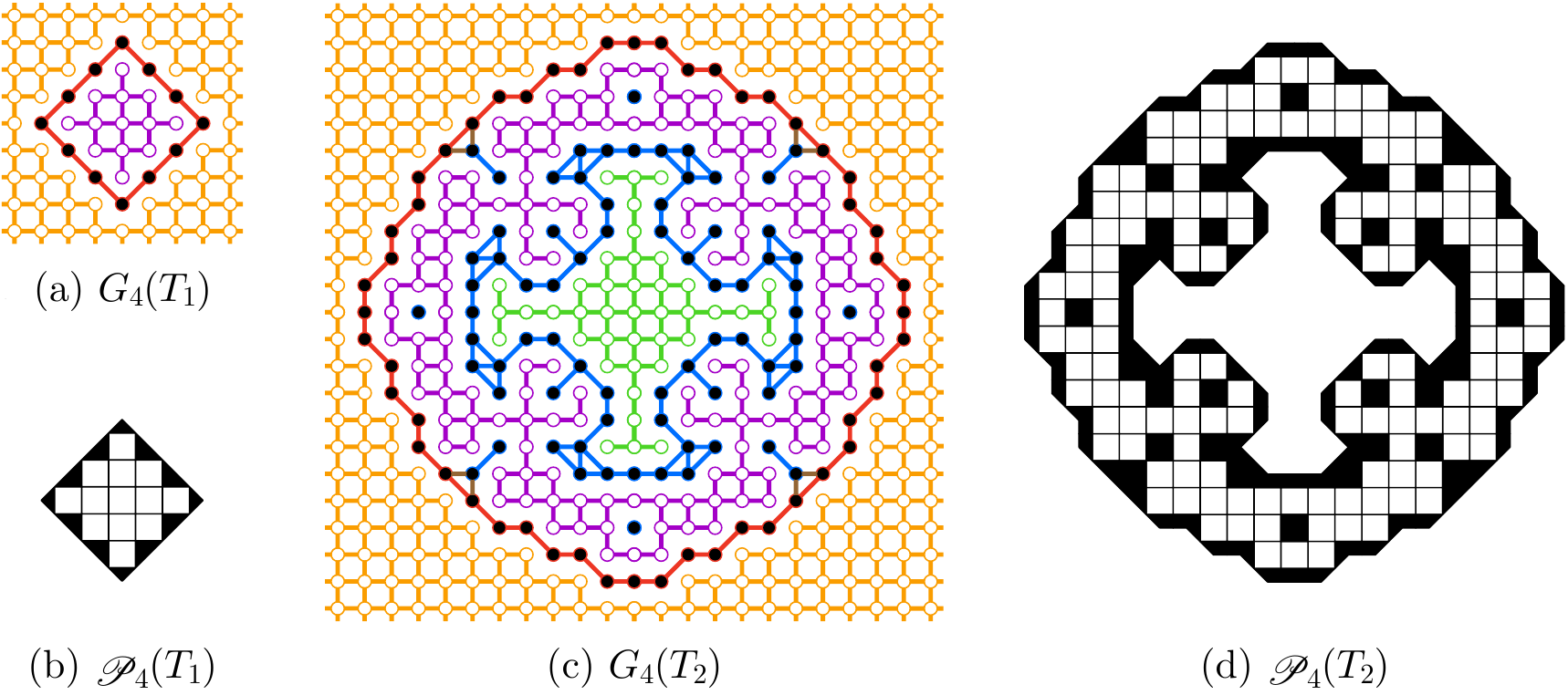}
\caption{Example characteristic graphs and polygonal tiles of two 4-connected components. For the respective component $T \in \{T_1, T_2\}$ in (a) and (c), we highlight $T$ (purple), $P_4^{ext}(T)$ (red), and $E_4(T)$ (orange). In (c) we also have $P_4^{int}(T_2)$ (blue), a finite $4$-connected component that $T_2$ surrounds (green), and 8 edges between $P_4^{ext}(T_2)$ and $P_4^{int}(T_2)$ (brown).}
\label{fig:exampletile}
\end{figure} 

\begin{figure}[ht]
\centering
\includegraphics[width=\textwidth]{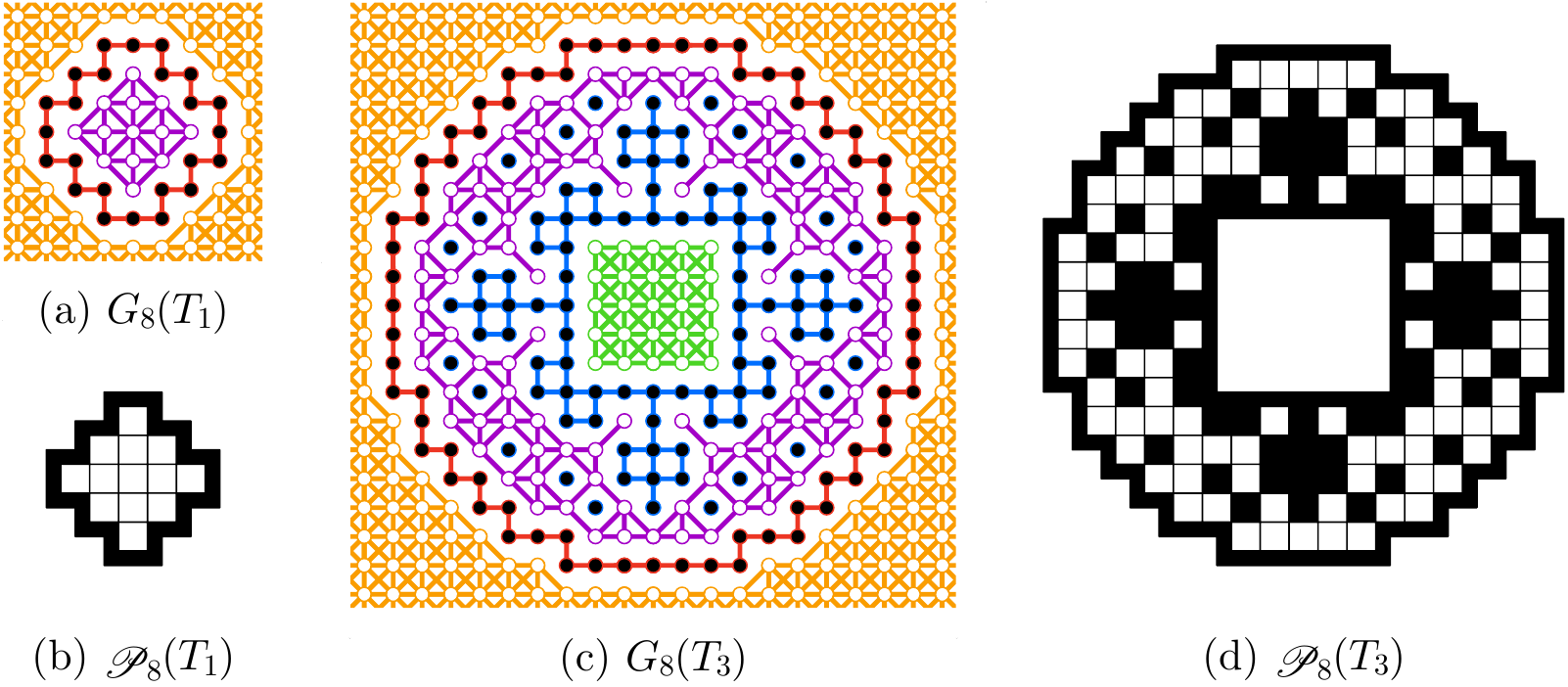}
\caption{Example characteristic graphs and polygonal tiles of two 8-connected components. For the respective component $T \in \{T_1, T_3\}$ in (a) and (c), we highlight $T$ (purple), $P_8^{ext}(T)$ (red), and $E_8(T)$ (orange). In (c) we also have $P_8^{int}(T_3)$ (blue), and a finite $8$-connected component that $T_3$ surrounds (green).}
\label{fig:exampletile2}
\end{figure} 

The desired maximality condition satisfied by polygonal tiles is a direct consequence of Definition \ref{def:pt} and the Voronoi property of $V_\al$ (Lemma \ref{lem:voronoi}). The set $\mc{M}$ that appears in the definition is specifically chosen to guarantee that nested polygonal tiles are disjoint. A more intuitive definition would be to only remove $V_\al(\tau)$ from $V_\al(T)$ if $\tau \in Z_\al(T)\setminus T$ is a (maximal) white $\al$-connected component in $G_\al(T)$. While it is intuitively clear that this would yield an equivalent definition, we do not prove this equivalence, and using Definition \ref{def:pt} simplifies the proof of Lemma \ref{lem:disjoint}. We now show that the polygonal tiles of each of the white connected components in an image are disjoint.

\begin{lemma}
\label{lem:disjoint}
     For any distinct white $\al$-connected components $T_1, T_2$ that appear in a binary image, $$\PP_\al(T_1) \cap \PP_\al(T_2) = \emptyset.$$
\end{lemma}

\begin{proof}
If $T_1$ is part of a finite white $\al$-connected component in $G_\al(T_2)$, or vice versa, then $\PP_\al(T_1)$ and $\PP_\al(T_2)$ are disjoint by construction. Therefore we only need to consider the case when $T_1 \subset E_\al(T_2)$ and $T_2 \subset E_\al(T_1)$. 

To leverage the result of Lemma \ref{lem:voronoi}, we want to show (1) $Z_\al(T_2) \subset E_\al(T_1)$ and (2) $Z_\al(T_1) \subset E_\al(T_2)$. For (1), let $H$ denote the colored, induced subgraph of $G_\al(T_1)$ that has the vertex set $Z_\al(T_2)$. We first show that all the vertices in $H$ are white. This will imply that $H$ is connected, and consequently $Z_\al(T_2) \subset E_\al(T_1)$. 

By definition, every point in $P_\al(T_1)$ 
has distance 1 to $T_1$ in the $d_\al$ metric. Therefore every point in $P_\al(T_1)$ is within distance 1 of $E_\al(T_2)$ by the assumption that $T_1 \subset E_\al(T_2)$. The digital Jordan curve theorem (Theorem \ref{thm:jct}) implies that there are no points in $H$ that are within distance 1 of $E_\al(T_2)$. Therefore $H \cap P_\al(T_1) = \emptyset$, which shows all the vertices in $H$ are white. 

Observe that the vertex set $Z_\al(T_2)$ is $\al$-connected by the digital Jordan curve theorem. Hence $H$ is a white $\al$-connected subgraph of $G_\al(T_1)$. Since $H$ contains $T_2$, the assumption $T_2 \subset E_\al(T_1)$ guarantees that $H$ contains a nonempty subset of $E_\al(T_1)$. Since $E_\al(T_1)$ is a maximal infinite white $\al$-connected component, connectivity of $H$ then implies that $H \subset E_\al(T_1)$, and hence $Z_\al(T_2) \subset E_\al(T_1)$. 

By symmetry, the proof of (2) is identical. Now, if $x \in \PP_\al(T_1) \subset V_\al(T_1)$, then Lemma \ref{lem:voronoi} implies $d_\al(x, Z_\al(T_1)) < d_\al(x, E_\al(T_1))$. By the fact (1) that we just proved, $d_\al(x, E_\al(T_1)) \leq d_\al(x, Z_\al(T_2))$, so combining these inequalities shows $d_\al(x, Z_\al(T_1)) < d_\al(x, Z_\al(T_2))$. Similarly, if $y \in \PP_\al(T_2) \subset V_\al(T_2)$, then Lemma \ref{lem:voronoi} and (2) imply that $d_\al(y, Z_\al(T_2)) < d_\al(y, E_\al(T_2)) \leq d_\al(y, Z_\al(T_1))$. Hence $x \neq y$, so $\PP_\al(T_1) \cap \PP_\al(T_2) = \emptyset$.
\end{proof}

\section{TWB and MGCS bounds}
\label{sec:bounds}

We now have all of the ingredients to prove our main result (Theorem \ref{thm:TWB}), and we first give an overview of the proof. Since the white $\al$-connected components in an image $I$ can be covered by a disjoint union of polygonal tiles (Lemma \ref{lem:disjoint}), upper bounding $\rho(I)$ reduces to upper bounding the white pixel density of the individual tiles in the covering. Each polygonal tile contains a single white $\al$-connected component of a known size, so upper bounding its density reduces to lower bounding its area. Since the exterior boundary of each polygonal tile is a simple lattice polygon (Lemma \ref{lem:simple}), we bound the area using Pick's theorem (Theorem \ref{thm:Picks}). Applying the lower bound on the size of the exterior site perimeter (Lemma \ref{lem:minpext}) finishes the argument. Recall that the functions $\sigma_\al(j)$, which have a simple closed form, denote the minimum site perimeter size of any $\al$-connected component of size $j$.

\begin{definition}
Define $f_\al: \R_{\geq 1} \to \big[\frac{2}{\al},1\big)$, $\displaystyle f_\al(j) := \frac{j}{j+\frac{\sigma_\al(j)}{2}-1}$.
\end{definition}

\begin{theorem}
\label{thm:TWB} For any finite set $S$ of $\al$-connected components,
    \begin{equation*}
        \Phi_\al(S) \leq \max_{T \in S}~ \rho(\PP_\al(T)) \leq \max_{\substack{j = |T|\\T \in S}}~ f_\al(j).
    \end{equation*}
\end{theorem}

\begin{proof}
For the first inequality, $\rho(I)$ is clearly upper bounded by the density of $I$ restricted to a union of polygonal tiles covering each white $\al$-connected component in $I$. These tiles are disjoint by Lemma \ref{lem:disjoint}. By a standard limit argument, it follows that the density of this union is a weighted average of the densities of the individual tiles. Therefore $\rho(I)$ is upper bounded by the maximum white pixel density of $\PP_\al(T)$ for some $T \in S$. 

For the second inequality, we want to obtain a sharp upper bound on the density of an arbitrary polygonal tile $\PP_\al(T)$. The density is given by $\rho(\PP_\al(T)) = \frac{|T|}{\text{area}(\PP_\al(T))}$, so it suffices to bound $\text{area}(\PP_\al(T))$. By Pick's theorem and Lemma \ref{lem:minpext}, if $|T|=j$,
\begin{equation}
\label{eq:areaP}
    \text{area}(\PP_\al(T)) \geq |T| + \frac{|P_\al^{ext}(T)|}{2}-1 \geq j + \frac{\sigma_\al(j)}{2}-1.
\end{equation}
For the first inequality in \eqref{eq:areaP}, we apply Pick's theorem without including the integer points $Z_\al(T) \setminus T$, which lie in the interior of $V_\al(T)$.
\end{proof}

It is natural to ask when the two inequalities in Theorem \ref{thm:TWB} are tight. The first inequality is tight if and only if the maximally dense polygonal tile(s) in $\PP_\al(S) := \{\PP_\al(T): T \in S\}$ can tile $\R^2$ by integer translations. If $\PP_\al(T)$ is one such maximally dense polygonal tile, it is easy to see that the second inequality is tight if and only if $P_\al^{int}(T) = \emptyset$.

For any finite component set $S$, Theorem \ref{thm:TWB} affords an easily computable upper bound on $\Phi_\al(S)$. Recalling that $S_\al(k)$ is the set of (canonical representatives of) all $\al$-connected components of size at most $k$, here we are most interested in the value of $\Phi_\al(k) := \Phi_\al(S_\al(k))$. Plugging this set into Theorem \ref{thm:TWB} shows that $\Phi_\al(k) \leq F_\al(k) := \max_{1 \leq j \leq k} f_\al(k)$, and inverting this result affords a lower bound on $C_\al(d)$. We record this straightforward calculation in the following corollary.

\begin{corollary}\label{cor:bounds}
~

\vspace{-3mm}
\noindent
$S_\al(k)$-TWB upper bounds $($for all $k \in \N)$$:$
\begin{align*}
    \Phi_8(k) &\leq F_8(k) \leq \frac{k}{\big(\sqrt{k}+1\big)^2},\\
    \Phi_4(k) &\leq F_4(k) \leq \frac{k}{k+\sqrt{2k-1}}.
\end{align*}
MGCS lower bounds $($for all $d \in (0,1])$$:$
\begin{align*}
    \C_8(d) &\geq F_8^{-1}(d) \geq \bigg\lceil\frac{d}{\big(1-\sqrt{d}\big)^2}\bigg\rceil,\\
    \C_4(d) &\geq F_4^{-1}(d) \geq  \begin{cases}1, & d \leq \frac{1}{2}\\
    \Big\lceil \frac{d^2 + d \sqrt{2d-1}}{(1-d)^2}\Big\rceil, & d > \frac{1}{2}
    \end{cases}~.
\end{align*}
\end{corollary}

\begin{proof}[Proof]

Plugging $S_\al(k)$ into Theorem \ref{thm:TWB} gives an upper bound on the solution to the $S_\al(k)$-TWB problem:
 \begin{equation}
 \label{eq:FG}
    \Phi_\al(k) \leq F(k) := \max_{1 \leq j \leq k} f_\al(k) \leq \begin{cases}
       \frac{k}{(\sqrt{k}+1)^2}, & \al = 8\\
       \frac{k}{k+\sqrt{2k-1}}, & \al = 4
    \end{cases}
\end{equation}
\noindent We then apply Lemma \ref{lem:inverses} to obtain
\begin{align*}
    \C_\al(d) \geq \min \{k \in \N:~\Phi_\al(k) \geq d\} &\geq F_\al^{-1}(d) := \min \{k \in \N:~F_\al(k) \geq d\}.
\end{align*}

Basic simplification shows that $F_\al^{-1}(d) = \min\big\{k \in \N:~ f_\al(k) \geq d\big\}$. Explicitly solving the equation $\frac{k}{(\sqrt{k}+1)^2} = d$ for $k \in \R_{\geq 1}$ gives rise to the simple lower bound on $F_8^{-1}(d)$ given in the corollary. Similarly, we obtain the simple lower bound on $F_4^{-1}(d)$ by solving $\frac{k}{k+\sqrt{2k-1}} = d$ for $k \in \R_{\geq 1}$. A Laurent expansion of the resulting functions (excluding the ceiling operator) around $d = 1$ yields the weaker lower bounds given in Corollary \ref{cor:MGCSsol}.
\end{proof}

\section{Constructions}
\label{sec:constructions}

\subsection{Infinite images}

To show that the bounds in Corollary \ref{cor:bounds} are tight, we generate (periodic) infinite images with white pixel density $F_\al(k)$ that only contain white $\al$-connected components of size at most $k$. Here, the key observation is that the spiral components $\mc{Q}_k$ and $\mc{R}_k$ (defined in Figure \ref{fig:QR}) always yield maximally dense polygonal tiles.

We say an open region $\mc{F} \subset \R^2$ (with closure $\overline{\mc{F}}$) tiles the plane with respect to a lattice $\mc{L}$ if the following two conditions hold:
\begin{enumerate}
    \item $\displaystyle \bigcup_{y \in \mc{L}} (\overline{\mc{F}}+y) = \R^2,$
    \item $\displaystyle \bigcap_{y \in \mc{L}} (\mc{F}+y) = \emptyset.$
\end{enumerate}

The 8-connected case is particularly simple. For all $k \in \N$, $\PP_8(\mc{R}_k)$ tiles the plane by translation with respect to a sublattice of $\Z^2$ (see Lemma \ref{lem:P8tiling} for a precise description of this tiling). This tiling is illustrated in Figure \ref{fig:tiling}(a). It is easy to verify that the density bound $F_8(k)$ is achieved (Corollary \ref{cor:F8achieved}). 

\begin{lemma}
\label{lem:P8tiling}
For each $k \in \N_{\geq 2}$, $\PP_8(\mc{R}_k)$ tiles the plane with respect to the lattice generated by $M^iB_k$.

\noindent Here, $M = \begin{pmatrix}
0 & 1 \\
-1 & 0
\end{pmatrix}$ represents a $\frac{\pi}{2}$ clockwise rotation, and $B_i = \begin{pmatrix}
2n+\lfloor\frac{i+1}{2}\rfloor & j+1 \\
-1 & 2n+\lfloor\frac{i}{2}\rfloor
\end{pmatrix}$.

\noindent The indices $i, j, n$, where $i \in \{0,1,2,3\}$ and $j,n \in \N$, are implicitly defined as follows: $k = t_i(n)+j$ where $j \in [t_{i+1}(n)-t_i(n)]$, and $t_4(n) := t_0(n+1)$. 
\end{lemma}

\begin{proof}
 Using the monotiling result of Girault-Beauquier and Nivat (\cite{GBN}; see Theorem \ref{thm:exacttile}), it is a finite computation to verify that the given tiling is correct in each of the four cases ($i \in \{0,1,2,3\}$). We illustrate the $i = 0$ case in Fig. \ref{fig:tiling} (a).
\end{proof}

\begin{corollary} \label{cor:F8achieved} For all $k \in \N$, $\Phi_8(k) = F_8(k).$
\end{corollary}

\begin{proof}
    The case $k=1$ is immediate, so we assume $k \geq 2$. In Lemma \ref{lem:P8tiling}, note that the area of $\PP_8(\mc{R}_k)$ is $\det(B_i) = k + \lceil2\sqrt{k} \rceil+1$. As a consequence, there always exists a $j \in [k]$ such that the white pixel density of tiling the plane by $\PP_8(\mc{R}_j)$ (via the Lemma \ref{lem:P8tiling} construction) corresponds with the density upper bound $F_8(k) = \max_{1 \leq j \leq k} ~\frac{j}{j+\lceil 2\sqrt{j}\rceil + 1}$. 
\end{proof}

The 4-connected case requires a bit more work. The diamond spiral $\mc{Q}_k$ has four distinct ``sides,'' and each side corresponds with a subfamily of connected components that all have similar polygonal tiles (see Fig. \ref{fig:QR}(a)). Along two of the sides, we find lattice tilings of the plane using $\PP_4(\mc{Q}_k)$ (Lemma \ref{lem:tiling4}, parts (1) and (2)). These tilings are illustrated in Figure \ref{fig:tiling}(b) and (c). We then calculate the values of $k$ where these tilings achieve density $F_4(k)$ (see Lemma \ref{lem:TWBsols}). 

In contrast, for values of $k$ corresponding with the other two sides of $\mc{Q}_k$, there exist instances of the $S_4(k)$-TWB problem where $\PP_4(\mc{Q}_{s_1(n)})$ or $\PP_4(\mc{Q}_{s_3(n)})$ are the uniquely densest polygonal tiles in $\PP_4(S_4(k))$. Recalling Remark \ref{rem:Q=sieben}, $\mc{Q}_{s_1(n)}$ and $\mc{Q}_{s_3(n)}$ are special cases of the component family $\mc{D}_{\ell,n}$ (see Definition \ref{def:ABD}). However, we find that $\PP_4(\mc{D}_{\ell,n})$ cannot tile the plane. We will use this result to show that $\Phi_4(k) \neq F_4(k)$ for these special instances of the $S_4(k)$-TWB problem. 

\begin{lemma} \label{lem:tiling4}
For all $n \in \N \text{:}$\\
\vspace{-0.5cm}
\begin{enumerate}
    \item If $k = s_1(n)+j$, where $j \in [n]$, then $\PP_4(\mc{Q}_k)$ tiles the plane with respect to the lattice generated by $\mc{V}_k := \begin{pmatrix}
    n & n+j\\
    -n-1 &n-j
\end{pmatrix}$.
    \item If $k = s_3(n-1)+j$, where $j \in [n]$, then $\PP_4(\mc{Q}_k) \cup \{-\PP_4(\mc{Q}_k)+(n,n)\}$ tiles the plane with respect to the lattice generated by $\mc{W}_k  := \begin{pmatrix}
    n-1 & 2n-1+j\\
    -n-1 &2n-1-j
\end{pmatrix}$.
    \item The region $\mathscr{P}(\mc{D}_{\ell,n})$ does not tile the plane.
\end{enumerate}
\end{lemma}

\begin{figure}[ht]
\centering
\includegraphics[width=0.9\textwidth]{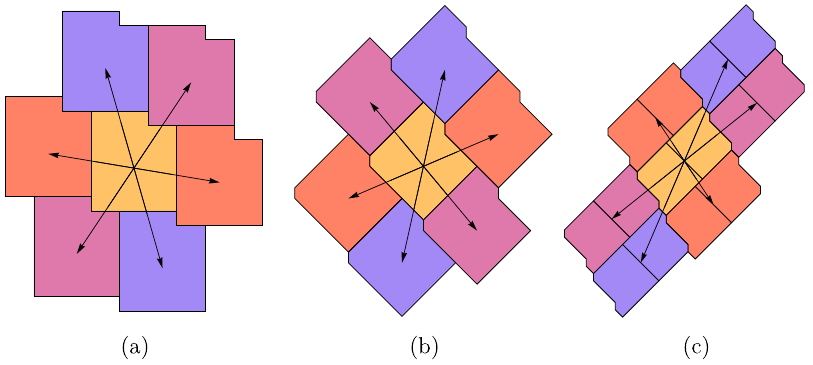}
\caption{The surroundings of each tile in the Lemma \ref{lem:P8tiling} and Lemma \ref{lem:tiling4} lattice tilings, with translation vectors shown.}
\label{fig:tiling}
\end{figure}

\begin{proof}
Using the monotiling result of Girault-Beauquier and Nivat (\cite{GBN}; see Theorem \ref{thm:exacttile}), it is a finite computation to verify that the given tilings are correct. We illustrate the tilings with respect to the lattices generated by $\mc{V}_k$ and $\mc{W}_k$ in Figs. \ref{fig:tiling} (b) and (c), respectively. 

To consider the tiling behavior of $\PP_4(\mc{D}_{\ell,n})$, we explicitly calculate its six boundary vertices:
\begin{equation} \label{eq:vpd}
    \{(0, -n), (-n, 0), (\ell - n, \ell), (\ell - n + 1, \ell), (\ell, \ell - n + 1), (\ell, \ell - n)\}
\end{equation}
There are three families of convex hexagons that can tile the plane \cite{Reinhardt}. Examining the vertex set \eqref{eq:vpd} reveals that $\PP_4(\mc{D}_{\ell,n})$ is a convex hexagon that is not in one of these three families.
\end{proof}

\begin{lemma}  \label{lem:TWBsols}
For all $k \in \mc{S}$, $\Phi_4(k) = F_4(k)$, where
\begin{equation*}
\mc{S} := \N\cap\bigcup_{n=1}^\infty \Big[s_0(n)-\frac{n}{2}, s_0(n)+\frac{n-1}{2}\Big] \cup \Big[s_2(n)-\frac{n}{2}, s_2(n)+\frac{n}{2}\Big].
\end{equation*}
    
For all $k \in \mc{U}$, $\Phi_4(k) < F_4(k)$, where
\begin{equation*}
    \mc{U} := \N_{\geq 2}\cap\bigcup_{n=1}^\infty  \Big[s_1(n), s_1(n)+\frac{n-1}{2}\Big] \cup \Big[s_3(n), s_3(n)+\frac{n}{2}\Big].
\end{equation*}
\end{lemma}

\begin{proof}
Fix a natural number $n$. By Lemma \ref{lem:tiling4} part (1), $\PP_4(\mc{Q}_k)$ tiles the plane (by translation) for $k \in [s_1(n)+1, s_1(n)+n]$. Note that $s_1(n)+n = s_2(n)$. The white pixel density of this tiling is $\rho(\PP_4(\mc{Q}_k)) = \frac{k}{\det(\mc{V}_k)} = \frac{k}{k+2n}$. By Lemma \ref{lem:closedform}, this density equals $F_4(k)$ for all $k \in [s_1(n)+\frac{n}{2}, s_2(n))$. Additionally, $F_4(k) = \frac{s_2(n)}{s_2(n)+2n}$ for all $k \in [s_2(n), s_2(n) + \frac{n+1}{2})$, and the specified tiling of the plane by $\PP_4(\mc{Q}_{s_2(n)})$ attains this density. This shows there are witnesses to $\Phi_4(k) = F_4(k)$ for all $k \in [s_1(n)+\frac{n}{2}, s_2(n)+\frac{n+1}{2}) = [s_2(n)-\frac{n}{2}, s_2(n)+\frac{n}{2}]$.

By Lemma \ref{lem:tiling4} part (2), $\PP_4(\mc{Q}_k)$ tiles the plane (by rotation and translation) for $k \in [s_3(n-1)+1, s_3(n-1)+n]$. Note that $s_3(n-1)+n = s_0(n)$. The white pixel density of this tiling is $\rho(\PP_4(\mc{Q}_k)) = \frac{2k}{\det(\mc{W}_k)} = \frac{k}{k+2n-1}$. By Lemma \ref{lem:closedform}, this density equals $F_4(k)$ for all $k \in [s_3(n-1)+\frac{n}{2}, s_0(n))$. Additionally, $F_4(k) = \frac{s_0(n)}{s_0(n)+2n-1}$ for all $k \in [s_0(n), s_0(n) + \frac{n}{2})$, and the specified tiling of the plane by $\PP_4(\mc{Q}_{s_0(n)})$ attains this density. Hence we obtain witnesses to $\Phi_4(k) = F_4(k)$ for all $k \in [s_3(n-1)+\frac{n}{2}, s_0(n) + \frac{n}{2}) = [s_0(n)-\frac{n}{2}, s_0(n) + \frac{n-1}{2}]$.

Fix integers $n \geq 2$ and $k \in [s_1(n), s_1(n)+\frac{n}{2})$. Up to isometries, $T := \PP_4(\mc{D}_{n,n})$ is the unique tile in $\PP_4(S_4(k))$ that has density $F_4(k)$ (by Lemmas \ref{lem:closedform} and \ref{lem:maxarea}). Assume for sake of contradiction that there exists a sequence of infinite images $I_m$ such that $\lim_{m \to \infty} \rho(I_m) = F_4(k)$ and each image contains white 4-connected components of size at most $k$. We want to show that this implies $T$ tiles the plane. 

By the convergence of $\rho(I_m)$ to $F_4(k)$, the polygonal tiles in $\PP_4(S_4(k)) \setminus T$ make no contribution to the limiting density, and hence we can assume without loss of generality that $I_m$ is a packing of isometric copies of $T$. If $j$ is an odd natural number, let $B(j)$ denote the closed $j \times j$ box centered at the origin. The convergence of $\rho(I_m)$ to $F_4(k)$ implies that $P_\ell := I_\ell \cap B(2\ell-1)$ is a sequence of packings of $T$ that cover a $1- \eps_\ell$ fraction of the area of $B(2\ell-1)$, where $\eps_\ell \to 0$. Fix $r \in \N$, and notice that for each sufficiently large $\ell$ there is an $r \times r$ region $b_\ell(r) \subset P_\ell$ such that at least a $1-2\eps_\ell$ fraction of $b_\ell(r)$ is covered by disjoint copies of $T$. The existence of $b_\ell(r)$ is implied by an averaging argument. Each copy of $T$ that intersects $b_\ell(r)$ is an integer translation of one of its four possible rotations, hence the configuration space of $b_\ell(r)$ is finite. Then since the covered fraction $1-2\eps_\ell$ converges to 1, $b_\ell(r)$ is fully covered (up to a set of measure zero) by disjoint copies of $T$ for all sufficiently large $\ell$. Since we imposed no restrictions on $r$, this shows that disjoint copies of $T$ cover arbitrarily large square regions of the plane. A standard extension argument then shows that $T$ must tile the whole plane (\cite[p. 69]{GOT}), which is a contradiction to Lemma \ref{lem:tiling4}, part (3).  

Similarly, fix integers $n \geq 1$ and $k \in  [s_3(n), s_3(n)+\frac{n+1}{2})$. Up to isometries, $T := \PP_4(\mc{D}_{n,n+1})$ is the unique tile in $\PP_4(S_4(k))$ that has density $F_4(k)$ (by Lemmas \ref{lem:closedform} and \ref{lem:maxarea}). The same argument as above shows that $\Phi_4(k) < F_4(k)$. 
\end{proof}

\subsection{Finite images}\label{sec:finiteMGCS}

In the finite case, we are primarily interested in the tightness of the bounds on the MGCS. Let $I$ be one of the witnesses to $\Phi_\al(k)$ that we generated in the previous section (i.e., one of the tilings given in Lemmas \ref{lem:P8tiling} or \ref{lem:tiling4}). Since $I$ is periodic with respect to a sublattice of $\Z^2$, we can generate arbitrarily large finite subimages of $I$ that have the same white pixel density as $I$ (Claim \ref{latticesq}). As illustrated in Example \ref{ex:tight}, taking these subimages to be large enough affords finite images where the Theorem \ref{thm:bounds} bounds are tight. 

\begin{claim}
\label{latticesq}
    If $I$ is an infinite image generated by tiling a fundamental domain $\mc{F}$ with respect to a lattice $\mc{L} \subseteq \Z^2$, for any $\ell \in \N$ we can construct a $\ell \det(\mc{L}) \times \ell \det(\mc{L})$ subimage of $I$ that has the same white pixel density as $I$.
\end{claim}
\begin{proof}
By definition,
\begin{equation*}
    I = \bigcup_{y \in \mc{L}} (\mc{F}+y).
\end{equation*}
Let $K^i \subset I$ denote the $i \times i$ square region with vertices $\{(-\frac{1}{2},-\frac{1}{2}),(-\frac{1}{2},i-\frac{1}{2}),(i-\frac{1}{2},-\frac{1}{2}),(i-\frac{1}{2},i-\frac{1}{2})\}$. The $(-\frac{1}{2},-\frac{1}{2})$ offset from $\Z^2$ is to ensure $K^i$ contains only full pixels. It is straightforward to show that $\ell \det(\mathcal{L})\cdot \Z^2 \subseteq \mathcal{L}$. Therefore $I$ is periodic with respect to $\ell\det(\mathcal{L})\cdot \Z^2$, hence
\begin{equation}
    I = \bigcup_{z \in \ell \det(\mc{L}) \cdot \Z^2} \big(K^{\ell \det(\mc{L})} + z\big).
\end{equation}
It follows that $\rho(K^{\ell \det(\mc{L})}) = \rho(I)$. 
\end{proof}

\begin{example}
\label{ex:tight}
    For the densities $d_j = \frac{j}{j+1}$, $j \geq 2$, we show how to explicitly construct finite images where Theorem \ref{thm:bounds} is tight.
\end{example}

\begin{figure}[ht]
    \centering
    \includegraphics[width=0.75\textwidth]{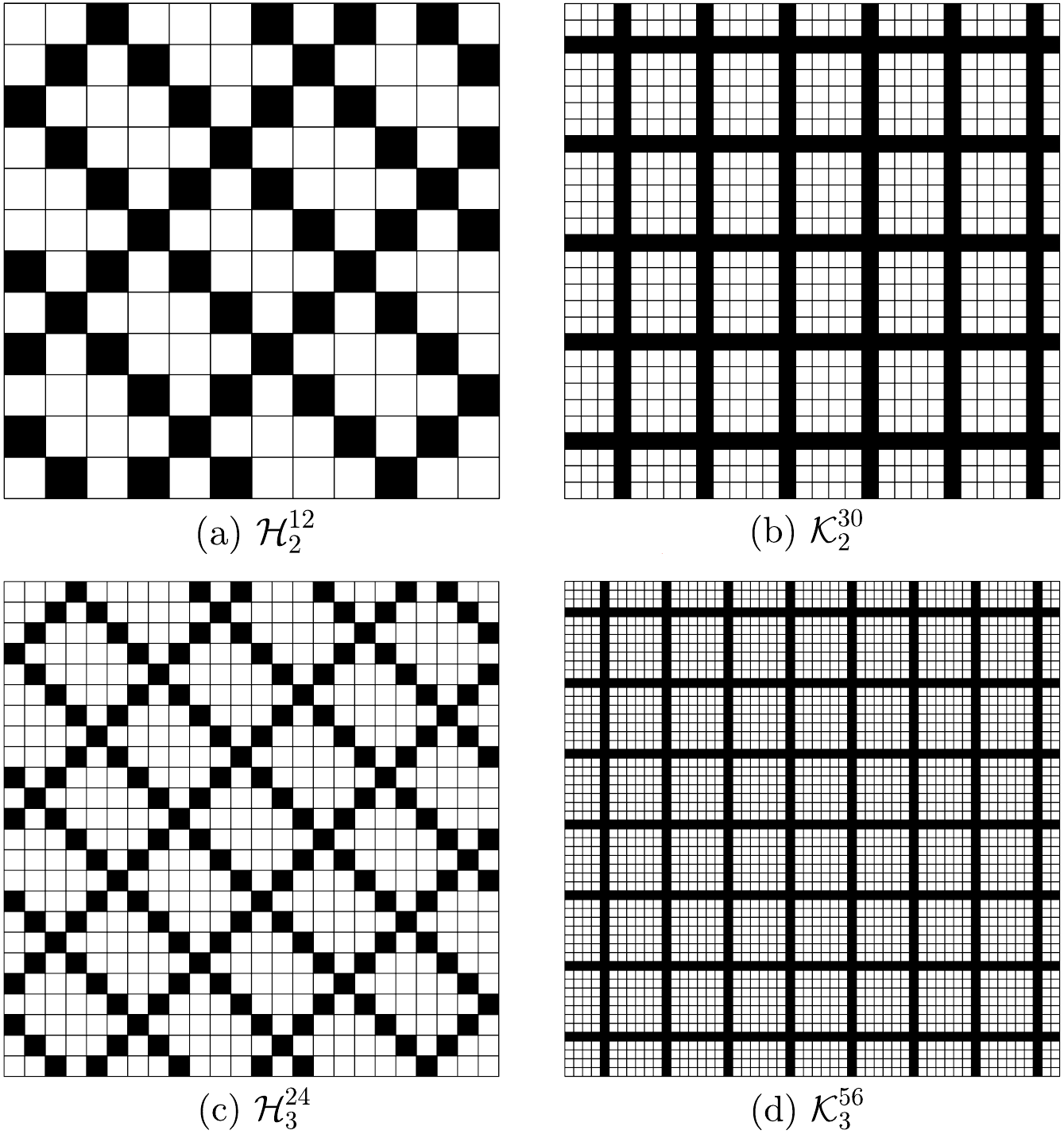}
    \caption{Taking a $4 \times 4$ grid of the subimages (a) and (b) yields two images that illustrate tightness of all the Theorem \ref{thm:bounds} bounds at $d = \frac{2}{3}$. Taking a $6 \times 6$ grid of the subimages (c) and (d) yields two images that illustrate tightness of all the Theorem \ref{thm:bounds} bounds at $d = \frac{3}{4}$.}
    \label{fig:dj-images}
\end{figure}

For each $j \geq 2$, we want to construct an $n_j \times n_j$ binary image that has white pixel density $d_j$ and maximal white 4-connected components of size $\C_4(d_j)$. Similarly, we want to construct an $m_j \times m_j$ binary image that has white pixel density $d_j$ and maximal white 8-connected components of size $\C_8(d_j)$. 

For the 4-connected case, we start with a tiling of the plane by $\PP_4(\mc{A}_{j, j+1})$ (see Remark \ref{rem:ABlattices}). Letting $\mc{L}(\mathscr{A})$ denote the lattice generated by $\mathscr{A}$, the infinite tiling can be formally written as
\begin{equation}
    \mc{G}_{j} := \bigcup_{y \in \mc{L}(\mathscr{A})} (\PP_4(\mc{A}_{j,j+1})+y), \hspace{6mm} \mathscr{A} := \begin{pmatrix}
j & -j-1 \\
j & j+1
\end{pmatrix}.
\end{equation}

Noting that $\det(\mathscr{A}) = 2j(j+1) =: u$, let $\mc{H}_j^i \subset \mc{G}_j$ denote the $i \times i$ square region with vertices $\{(-\frac{1}{2},-\frac{1}{2}),(-\frac{1}{2},i-\frac{1}{2}),(i-\frac{1}{2},-\frac{1}{2}),(i-\frac{1}{2},i-\frac{1}{2})\}$. By Claim \ref{latticesq}, $\rho(\mc{H}_j^{\ell u})= \rho(\mc{G}_j) = d_j$ for any $\ell \in \N$. To ensure $F_4^{-1}(r_4) = \C_4(d_j) = 2j^2$, we need to make $n_j$ large enough such that $r_4(d_j, n_j, n_j) > F_4(2j^2-1) = \frac{2j^2-1}{2j^2+2j-1}$. An easy calculation shows that $n_j \geq (4j^2-2)(j+1)$ suffices. To satisfy the constraint $u~ | ~n_j$, we set $n_j := 4j^2(j+1)=2ju$. By construction, $\mc{H}_j^{n_j}$ is then an $n_j \times n_j$ image with density $d_j$, and its largest white 4-connected components have size $F_4^{-1}(r_4(d_j, n_j, n_j))$ (see Figs. \ref{fig:dj-images} (a) and (c)). 

In the 8-connected case, the lattice tiling we consider is
\begin{equation}
    \mc{J}_{j} := \bigcup_{y \in \mc{L}(\mathscr{R})} (\PP_8(\mc{R}_{2j(2j+1)})+y), \hspace{6mm} \mathscr{R} := \begin{pmatrix}
2j+1 & 0 \\
0 & 2j+2
\end{pmatrix}.
\end{equation}
Observe that $\PP_8(\mc{R}_{2j(2j+1)})$ is a rectangle with area $\det(\mathscr{R}) = (2j+1)(2j+2) =: v$. For the sake of comparison, we choose to construct a square subimage in the same way as the 4-connected case. Let $\mc{K}_j^i \subset \mc{J}_j$ denote the $i \times i$ square region with vertices $\{(-\frac{1}{2},-\frac{1}{2}),(-\frac{1}{2},i-\frac{1}{2}),(i-\frac{1}{2},-\frac{1}{2}),(i-\frac{1}{2},i-\frac{1}{2})\}$. By Claim \ref{latticesq}, $\rho(\mc{K}_j^{\ell v})= \rho(\mc{J}_j) = d_j$ for any $\ell \in \N$. To ensure $F_8^{-1}(r_8) = \C_8(d_j) = 2j(2j+1)$, we need to make $m_j$ large enough such that $r_8(d_j, m_j, m_j) > F_8(2j(2j+1)-1) = \frac{2j(2j+1)-1}{2j(2j+1)+4j+1}$. An easy calculation shows that $m_j \geq (4j^2+4j-1)(2j+1)$ suffices. To satisfy the constraint $v~ | ~m_j$, we set $m_j := (4j^2+4j)(2j+1) = 2jv$. Hence $\mc{K}_j^{m_j}$ is an $m_j \times m_j$ image with density $d_j$, and its largest white 8-connected components have size $F_8^{-1}(r_8(d_j, m_j, m_j))$ (see Figs. \ref{fig:dj-images} (b) and (d)). 

\section{Additional solutions in the 4-connected case}
\label{sec:4conncomp}

In Theorem \ref{thm:TWBsol}, $k = 3 = s_3(1)$ and $k = 6 = s_1(2)$ are the smallest integers where the value of $\Phi_4(k)$ is unknown. The challenge that arises in these cases is that $\Phi_4(k) < F_4(k)$ (by Lemma \ref{lem:TWBsols}), so we need to use a more refined argument to construct an upper bound on $\Phi_4(k)$. We find that $\Phi_4(6)$ is actually equal to $\Phi_4(5) = \frac{5}{8}$.

\begin{figure}[ht]
\centering
\includegraphics[width=0.78\textwidth]{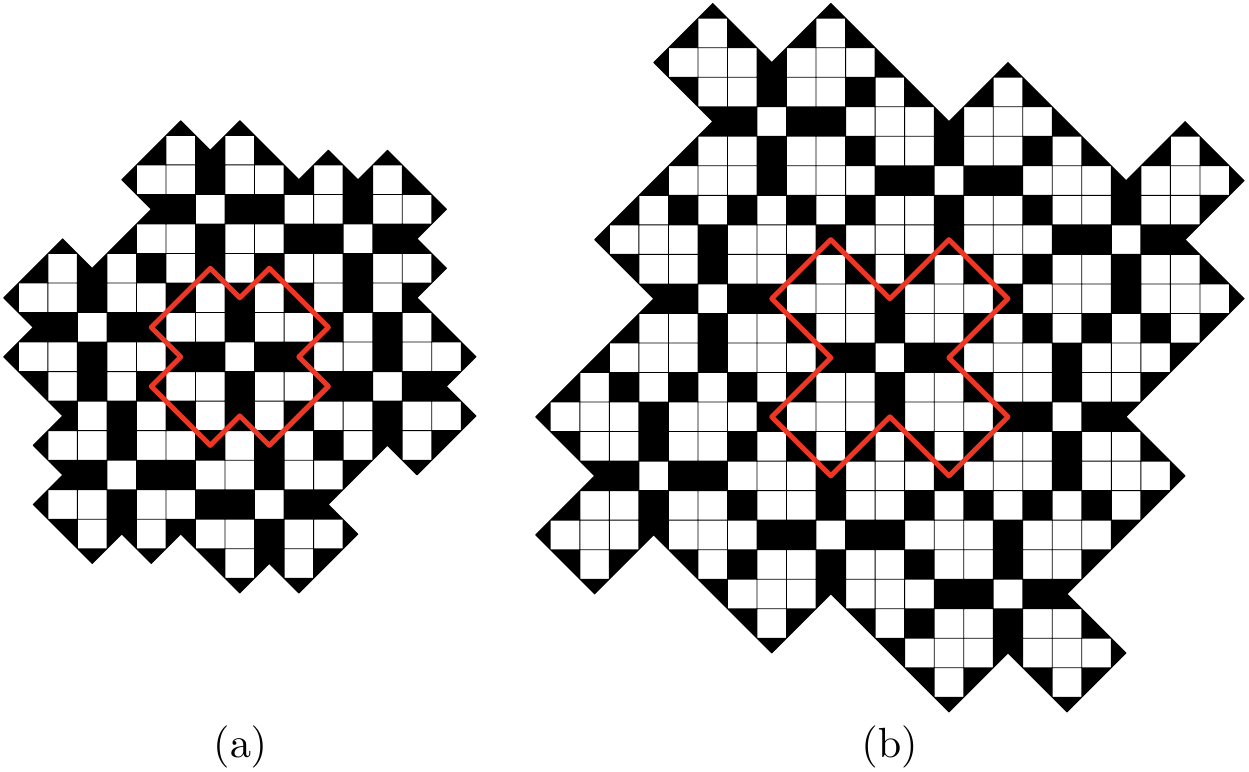}
\caption{Tiling witnesses to $\Phi_4(3)$ (a) and $\Phi_4(6)$ (b) with fundamental domains outlined in red.}
\label{fig:36tilings}
\end{figure}

\begin{theorem}
\label{thm:36}
    \begin{align*}
        \Phi_4(3) &= \frac{13}{24}\\
        \Phi_4(6) &= \frac{5}{8}
    \end{align*}
\end{theorem}

\begin{proof}
Figure \ref{fig:36tilings} shows witnesses to the claimed density values, hence it is sufficient to prove $\Phi_4(3) \leq \frac{13}{24}$ and $\Phi_4(6) \leq \frac{5}{8}$. Let $I_3$ denote any infinite image that only contains white $4$-connected components in $S_4(3)$ and let $I_6$ denote any infinite image that only contains white $4$-connected components in $S_4(6)$. We want to show $I_3$ and $I_6$ can be partitioned into regions with density at most $\frac{13}{24}$ and $\frac{5}{8}$, respectively. For the remainder of this proof, we refer to white $4$-connected components as just components, and we consider two components or tiles to be distinct if and only if there is no isometry of $\Z^2$ that maps one to the other. 

Since distinct polygonal tiles have empty intersection (Lemma \ref{lem:disjoint}), there is a trivial partition of $I_3$ and $I_6$ into black regions and translations of polygonal tiles generated by $S_4(3)$ and $S_4(6)$. Up to isometries of $\Z^2$, there is a unique polygonal tile in $\PP_4(S_4(3))$ and $\PP_4(S_4(6))$ that attains the respective $F_4(3) = \frac{6}{11} > \frac{13}{24}$ and $F_4(6) = \frac{12}{19} > \frac{5}{8}$ upper bounds on the tiling density (Lemma \ref{lem:maxarea}). These polygonal tiles are generated by $\mc{Q}_3$ and $\mc{Q}_6$, respectively, which are the size 3 and size 6 components that appear in Figure \ref{fig:36tilings}. All other tiles in $\PP_4(S_4(3))$ and $\PP_4(S_4(6))$ have densities less than or equal to $\frac{1}{2}$ and $\frac{5}{8}$, respectively. 
Since $\mc{Q}_6$ is the union of two copies of $\mc{Q}_3$, we find that the analysis of how $\mc{Q}_3$ behaves in the $S_4(3)$-TWB problem is similar to the analysis of how $\mc{Q}_6$ behaves in the $S_4(6)$-TWB problem.
    
\begin{figure}[ht]
\centering
\includegraphics[width=\textwidth]{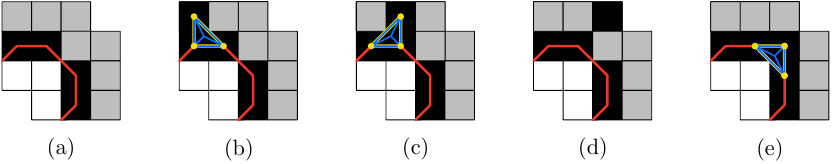}
\caption{Neighborhoods of $\mc{Q}_3$ (and $\mc{Q}_6$) around the apex point. Gray cells have unspecified color, and subfigure (a) shows the full apex neighborhood in gray, with the apex point at the center. By symmetry and component size constraints, (b), (c), (d), and (e) are the only four neighborhoods we need to consider. The small triangles are discussed in the text.}
\label{fig:nose}
\end{figure} 

We call the gray cells in Figure \ref{fig:nose} (a) the \textit{apex neighborhood}. This neighborhood consists of 7 cells that are 4-connected, so they cannot all be white. By symmetry, this leads to only four cases we need to consider, which are illustrated in Figures \ref{fig:nose} (b), (c), (d), and (e). Any copy of $\mc{Q}_3$ or $\mc{Q}_6$ that appears in $I_3$ or $I_6$ must have at least one of these apex neighborhoods. 

In Figures \ref{fig:nose} (b), (c), and (e), a black 3-cycle is formed in the image graph (highlighted in yellow). It is clear from these figures that the triangular region bounded by this 3-cycle is not contained in $\PP_4(T)$ for any white component $T$ appearing in $I_3$ or $I_6$. Since each triangle generated in this way can share an edge with at most three surrounding tiles, we split the black triangle into three smaller triangular regions of area $\frac{1}{6}$ (outlined in blue in Fig. \ref{fig:nose}), which share a vertex at the centroid of the triangle. If an isometric copy of the tile $\PP_4(\mc{Q}_3)$ in $I_3$ or $\PP_4(\mc{Q}_6)$ in $I_6$ shares an edge with one of these small triangular regions (as shown in Fig. \ref{fig:nose} (b), (c), and (e)), join the small triangular region with the respective tile. This operation results in modified tiles---denoted $\tilde{\PP}_4(\mc{Q}_3)$ and $\tilde{\PP}_4(\mc{Q}_6)$---that each contain at least one extra black triangle with area $\frac{1}{6}$. Hence $\rho(\tilde{\PP}_4(\mc{Q}_3)) \leq \frac{3}{\frac{11}{2}+\frac{1}{6}} = \frac{9}{17}$ and $\rho(\tilde{\PP}_4(\mc{Q}_6)) \leq \frac{6}{\frac{19}{2}+\frac{1}{6}} = \frac{18}{29}$. This shows that any copy of $\mc{Q}_3$ with apex neighborhood (b), (c), or (e) in $I_3$ contributes density at most $\frac{9}{17} < \frac{13}{24}$ to the construction. Similarly, any copy of $\mc{Q}_6$ with apex neighborhood (b), (c), or (e) that appears in $I_6$ contributes density at most $\frac{18}{29} < \frac{5}{8}$ to the construction. 

In Fig. \ref{fig:nose}, neighborhood (d) is special because it need not generate any black 3-cycles. In the case of the $S_4(3)$-TWB problem, the unique way to fill in the remaining gray squares in (d) to avoid creating a (b), (c), or (e)-type neighborhood is shown in Figure \ref{fig:36nbhd} (a). In the $S_4(6)$-TWB problem, there is one additional way to complete the (d) neighborhood.
\begin{figure}[ht]
\centering
\includegraphics[width=0.73\textwidth]{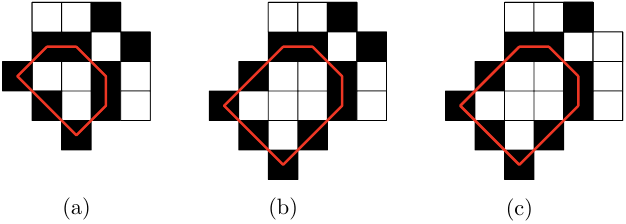}
\caption{The unique apex neighborhoods of $\mc{Q}_3$ and $\mc{Q}_6$ that incur no black triangles.}
\label{fig:36nbhd}
\end{figure}  

In the $S_4(3)$-TWB problem, let the term \textit{apical cell} denote the unique white cell that touches the apex point of a copy of $\mc{Q}_3$ with the Fig. \ref{fig:36nbhd} (a) neighborhood. For any apical cell $A$ that appears in $I_3$, let $\tilde{\PP}_4(A)$ be the union of $\PP_4(A)$ and all polygonal tiles generated by isometric copies of $\mc{Q}_3$ that have $A$ as their apical cell. There is always at least one such copy of $\mc{Q}_3$, otherwise $A$ would not be an apical cell. Since each apical cell can touch at most 4 apex points, $\rho(\tilde{\PP}_4(A)) \leq \frac{1+4\cdot 3}{2+4\cdot \frac{11}{2}} = \frac{13}{24}$. 

Therefore we have shown that there exists a partition of $I_3$ such that every copy of $\mc{Q}_3$ is part of a region $\tilde{\PP}_4(\mc{Q}_3)$ or $\tilde{\PP}_4(A)$, where $\rho(\tilde{\PP}_4(\mc{Q}_3)) \leq \frac{9}{17}$ and $\rho(\tilde{\PP}_4(A)) \leq \frac{13}{24}$. Importantly, the regions $\tilde{\PP}_4(\mc{Q}_3)$ and $\tilde{\PP}_4(A)$ do not intersect with any other polygonal tiles in $I_3$, hence any additional area in $I_3$ has white pixel density at most $\frac{1}{2}$. This completes the proof that $\Phi_4(3) \leq \frac{13}{24}$. Since there is only one choice of $\tilde{\PP}_4(A)$ that attains the density $\frac{13}{24}$, we note that the witness to $\Phi_4(3)$ (see Fig. \ref{fig:36tilings}(a)) is unique up to isometries. 

The $S_4(6)$-TWB analysis is slightly more involved, as the cases illustrated in Figs. \ref{fig:36nbhd} (b) and (c) both need to be considered. The Fig. \ref{fig:36nbhd} (b) case is nearly identical to the Fig. \ref{fig:36nbhd} (a) case above; we again use the term \textit{apical cell} to denote the unique white cell that touches the apex point of a copy of $\mc{Q}_6$ with the Fig. \ref{fig:36nbhd} (b) neighborhood. For any apical cell $B$ appearing in $I_6$, let $\tilde{\PP}_4(B)$ be the union of $\PP_4(B)$ and all polygonal tiles generated by isometric copies of $\mc{Q}_6$ that have $B$ as their apical cell. Since each apical cell can touch at most 4 apex points, $\rho(\tilde{\PP}_4(B)) \leq \frac{1+4\cdot 6}{2+4\cdot \frac{19}{2}} = \frac{5}{8}$. 

\begin{figure}[ht]
\centering
\includegraphics[width=0.8\textwidth]{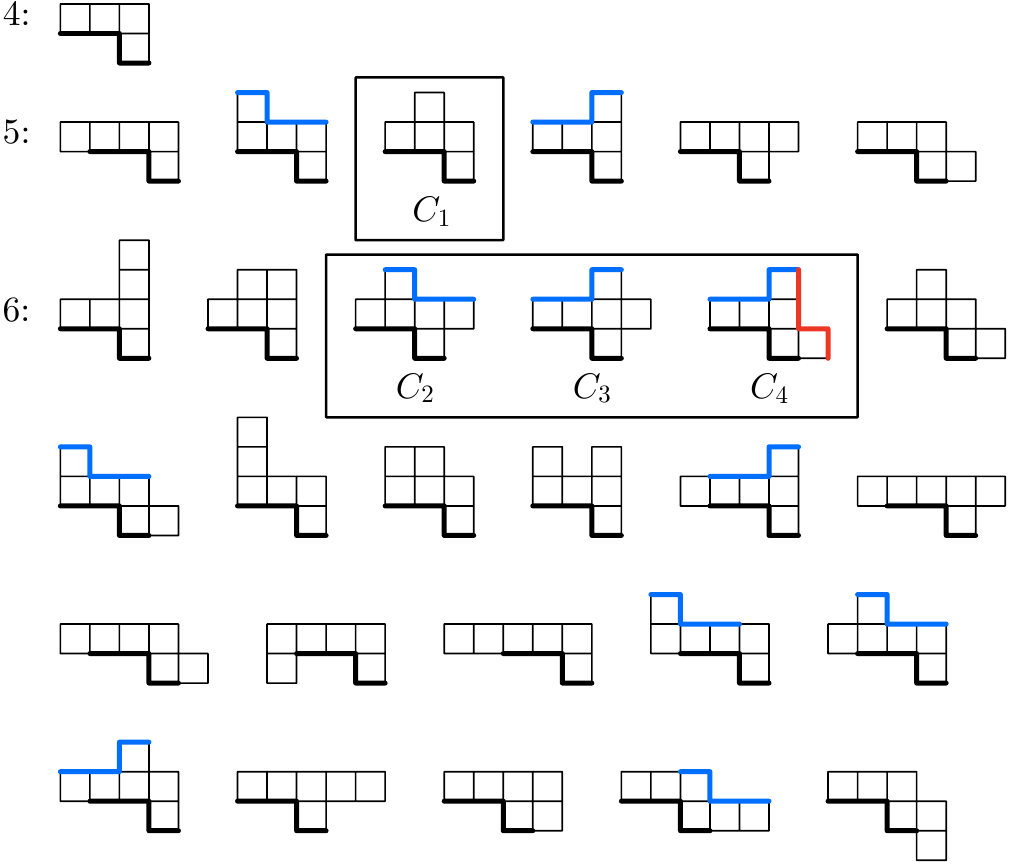}
\caption{The set of feasible apical components in $S_4(6)$ (up to isometries of $\Z^2$). Counting the site perimeters shows that the boxed components are the candidates for maximizing $\rho(\tilde{\PP}_4(C))$.}
\label{fig:enumerate}
\end{figure} 

Let the term \textit{apical component} denote the unique white component that touches the apex point of a copy of $\mc{Q}_6$ with a Fig. \ref{fig:36nbhd} (c) neighborhood. Observe that each apical component must contain a particular contour shape. For any apical component $C$ appearing in $I_6$, let $\tilde{\PP}_4(C)$ be the union of $\PP_4(C)$ and all polygonal tiles generated by isometric copies of $\mc{Q}_6$ that have $C$ as their apical component. We then upper bound the density of $\tilde{\PP}_4(C)$ by enumerating all possible choices (Fig. \ref{fig:enumerate}). By Pick's theorem (Theorem \ref{thm:Picks}), counting the site perimeter of each apical component reduces the density computation to just three cases.

\begin{align}
    \rho(\tilde{\PP}_4(C_1)) &= \frac{5+6}{5+\frac{9}{2}-1+\frac{19}{2}} = \frac{11}{18} < \frac{5}{8}\\
    \rho(\tilde{\PP}_4(C_2)) = \rho(\tilde{\PP}_4(C_3)) &\leq \frac{6+2 \cdot 6}{6+\frac{10}{2}-1+2\cdot \frac{19}{2}} = \frac{18}{29} <\frac{5}{8}\\
    \rho(\tilde{\PP}_4(C_4)) &\leq \frac{6+3 \cdot 6}{6+\frac{11}{2}-1+3\cdot\frac{19}{2}} = \frac{8}{13} < \frac{5}{8}
\end{align}

Finally, we have shown that each isometric copy of $\mc{Q}_6$ in $I_6$ is part of a disjoint region of the type $\tilde{\PP}_4(\mc{Q}_6)$, $\tilde{\PP}_4(B)$, or $\tilde{\PP}_4(C)$. Moreover, these regions are disjoint from all other polygonal tiles in $I_6$. The respective densities of these three regions is at most $\frac{18}{29}$, $\frac{5}{8}$, and $\frac{18}{29}$. Since all remaining polygonal tiles in $I_6$ have density at most $\frac{5}{8}$, it follows that $\Phi_4(6) \leq \frac{5}{8}$.
\end{proof}

It is desirable to extend the argument we used in Theorem \ref{thm:36} to construct upper bounds on $\Phi_4(s_1(n))$ and $\Phi_4(s_3(n))$ for all $n \in \N$, but this is beyond the scope of this work. More generally, the key elements that are needed to solve the remaining cases of the $S_4(k)$-TWB problem, and hence the 4-connected MGCS problem, are to identify maximally dense polygonal tiles in $\PP_4(S_4(k))$ and to characterize their tiling behavior.

\section{Future Directions}
\label{sec:future}

\noindent We conclude by asking some questions that suggest possible future directions.

\begin{enumerate}
    \item In many real-world applications, it is good to have an objective choice for what constitutes a small connected component. For example, binary images can be denoised by removing ``small'' components if the size threshold for ``small'' is well-chosen \cite{CDY}. The MGCS is a connected component size statistic that depends only on the density and dimensions of an image. Recall that we have computed the MGCS exactly for a large range of parameters, and we have given quite tight approximations in the remaining cases (see Theorem \ref{thm:bounds} and Corollary \ref{cor:MGCSsol}). We therefore ask whether the MGCS---or some simple modification of it---provides a useful connected component size threshold in any image processing applications.

    \item Recall that we have exactly computed $\Phi_8(k)$ for all $k \in \N$. We have also computed $\Phi_4(k)$ for a set of integers $k$ with natural density $\frac{1}{2}$, and we leave open the problem of computing $\Phi_4(k)$ for the remaining integers $k$.

    As this problem appears to be quite hard (e.g., see the proof of Theorem \ref{thm:36}), we note two steps that may be easier. First, we ask whether one can finish characterizing the integers $k$ such that $\Phi_4(k) = F_4(k)$. Second, Theorem \ref{thm:36} shows that $\Phi_4(s_1(2))=\Phi_4(s_0(2))$, and it is also trivially the case that $\Phi_4(s_1(1))= \Phi_4(s_0(1))$ (the sequences $s_i(n)$ are defined in Lemma \ref{lem:maxarea}). If this pattern continues, it would resolve the value of $\Phi_4(k)$ for an additional set of integers $k$ with natural density $\frac{1}{8}$. This makes it desirable to exactly compute $\Phi_4(s_1(n))$. 
    
    \item The Tiling with Boundaries (TWB) problem was defined for arbitrary connected component sets. While the first inequality in Theorem \ref{thm:TWB} concerns the general case, we otherwise focused our attention on the $S_\al(k)$-TWB problem, where $S_\al(k)$ is the set of all $\al$-connected components of size at most $k$. However, it is interesting to consider other component sets. 
    
    One can ask how much the relative orientations of components affects the solution to the TWB problem in the general case. For example, for any two fixed white $\al$-connected components $T_1, T_2 \in S_\al(k)$, what is the maximum value of $|\Phi_\al(\{T_1, T_2\}) - \Phi_\al(\{T_1, T_2'\})|$, where $T_2'$ is an isometric transformation of $T_2$?
    
    \item What results can be obtained about suitable generalizations of the TWB and MGCS problem to other lattices in $\R^n$? Even in $\R^2$, connected components on the hexagonal lattice $A_2$ (both the so-called polyhexes and polyiamonds) are routinely studied in the context of percolation. In a tiling of the plane by regular hexagons, connected components (unions of edge-connected hexagons called polyhexes) and their boundaries are both 6-connected. This suggests that the polyhex version of the TWB problem is quite natural.  
\end{enumerate}

\section*{Acknowledgements}

I am grateful to Noah Stephens-Davidowitz for numerous helpful conversations and providing invaluable feedback on various drafts of this paper. I would also like to thank Veit Elser for helpful conversations, and my colleagues in Cornell's Center for Applied Math for their feedback. Finally, I would like to thank Alma Dal Co, who introduced me to connected components in binary images.

\appendix
\section{8-connected component isoperimetry}

\begin{lemma}
\label{lem:boxbound}
    If $T$ is an $8$-connected component, let $R(T)$ denote the $4$-connected bounding rectangle of $T$, i.e. the smallest $4$-connected rectangle that surrounds $T$. Then $$|P_8(T)| \geq |P_8^{ext}(T)| \geq |R(T)|.$$
\end{lemma}

Note that the inequality $|P_8(T)| \geq |P_8^{ext}(T)|$ is trivial, but $|P_8(T)| \geq |R(T)|$ is the result used in the proof of Lemma \ref{lem:minperim8}.

\begin{proof}
    To prove this result, we expand on the projection argument given in the proof of Theorem 8 in \cite{A06}.

    \begin{figure}[ht]
    \centering
    \includegraphics[width=\textwidth]{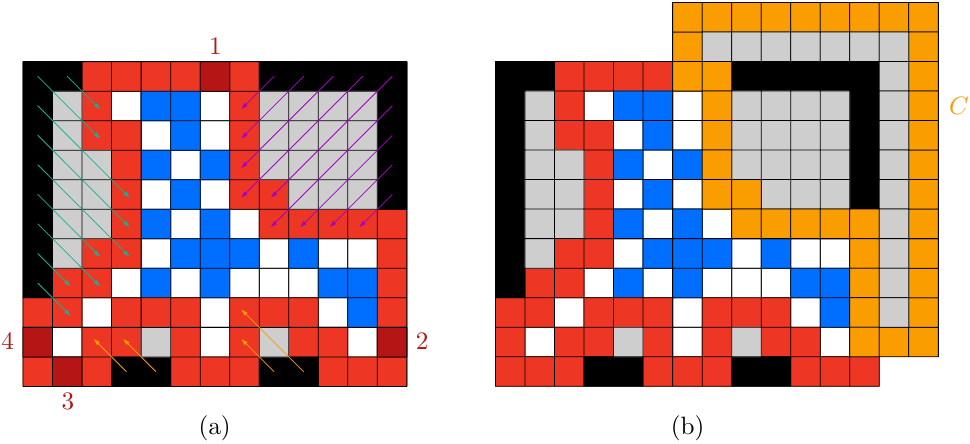}
    \caption{An example illustrating the projection of $R(T)$ into $P_8^{ext}(T)$. In (a), cells that move under the projection are marked in black---their final positions are shown by arrows. The red cells are in $P_8^{ext}(T)$ and blue cells are in $P_8^{int}(T)$. In (b), we highlight a choice of the simple cycle $C$ described in the proof.}
    \label{fig:P8_Proj}
    \end{figure}

    By the minimality of $R(T)$, there must exist at least one point along each of the four sides of $R(T)$ that is $d_4$-adjacent to a point in $T$. Traversing clockwise around $R(T)$, label the last point that is $d_4$-adjacent to $T$ along the top, right, bottom, and left sides by $1,2,3,$ and $4$, respectively.

    Now, let $\proj$ denote a map that projects the points in $R(T)$ until they are $d_8$-adjacent to $T$. Specifically, the points in $R(T)$ between 1 and 2 are projected $45^{\circ}$ down-left, the points between 2 and 3 are projected $45^{\circ}$ up-left, the points between 3 and 4 are projected $45^{\circ}$ up-right, and the points between 4 and 1 are projected $45^{\circ}$ down-right (see Fig. \ref{fig:P8_Proj}). We want to show that the codomain of $\proj$ is $P_8^{ext}(T)$ and that it is injective.

   By construction, the points $1,2,3,$ and $4$ are in the set $P_8^{ext}(T)$. Since $P_8^{ext}(T)$ is a simple cycle (Lemma \ref{lem:simple}), for all $i \in \{1,2,3,4\}$ let $\gamma_{i, i+1} \subset P_8^{ext}(T)$ be the unique 4-connected set of points between $i$ and $i+1$ (endpoints included), traveling clockwise. Similarly, let $R_{i,i+1} \subset R(T)$ be the unique 4-connected set of points between $i$ and $i+1$ (endpoints included), traveling clockwise. 

   To show show the codomain of $\proj$ is $P_8^{ext}(T)$, by rotational symmetry it suffices to consider the action of $\proj$ on $x \in R_{1,2}$. The rotated $L$-shape of $R_{1,2}$ allows us to construct a set $\Gamma_{1,2}$ whose points are all ``up-right'' of $R_{1,2}$ and have the property that $C = \gamma_{1,2} \cup \Gamma_{1,2}$ is a simple 4-connected cycle (see Fig. \ref{fig:P8_Proj}). By construction, the points $R(T)\setminus R_{1,2}$ are in the exterior of $C$. Each point $x \in R_{1,2}$ is projected down-left along a $45^{\circ}$, 8-connected ray. Applying the digital Jordan curve theorem to $C$ (Theorem \ref{thm:jct}), the projection ray of $x$ must contain a point in $\gamma_{1,2}$ or $\Gamma_{1,2}$. Since all points in $\Gamma_{1,2}$ are up-right of $x$, the projection ray of $x$ cannot contain a point in $\Gamma_{1,2}$, hence it must contain a point in $\gamma_{1,2}$. By construction, each point in $\gamma_{1,2}$ is $d_8$-adjacent to $T$ and \textit{no} points in the interior of $C$ are $d_8$-adjacent to $T$, so $\proj(x) \in \gamma_{1,2}$. Furthermore, since the projection rays are parallel, $\proj: R_{1,2} \to \gamma_{1,2}$ is injective.

 Since both $P_8^{ext}(T)$ and $R(T)$ are simple 4-connected cycles and the points $1,2,3,$ and $4$ are unique, $\gamma_{i, i+1} \cap \gamma_{j,j+1} = R_{i, i+1} \cap R_{j,j+1} = \{i, i+1\} \cap \{j,j+1\}$ if $i \neq j$. Of course, if $i = 4$ we set $i + 1 \equiv 1$. Since the intersection points $1,2,3,$ and $4$ are fixed under the projection, the injectivity of $\proj: R_{i,i+1} \to \gamma_{i,i+1}$ for $i \in [4]$ implies that $\proj: R(T) \to P_8^{ext}(T)$ is injective. 
\end{proof}

\begin{claim}
\label{claim1}
    Let $R$ be a rectangle with integer side lengths $\ell$ and $w$. Then
    \begin{enumerate}
        \item If $\ell + w = n$, $n \geq 2$, the area of $R$ is uniquely maximized (up to switching $\ell$ and $w$) if $\ell = \lceil\frac{n}{2} \rceil$ and $w = \lfloor \frac{n}{2}\rfloor$. 
        \item If $R$ has area $a$, the perimeter of $R$ is at least $2\big\lceil 2 \sqrt{a}\big\rceil$.
    \end{enumerate}
\end{claim}

\begin{proof}
    The AM-GM inequality implies $\ell w \leq (\frac{\ell+w}{2})^2 = \frac{n^2}{4}$, which shows (1) for integers $\ell, w$.

    \noindent Statement (2) is also straightforward, but we show its proof for completeness. If the perimeter and area of $R$ are $p$ and $a(p)$, respectively, the first statement implies
    \begin{equation}
    \label{eq:a(p)}
        a(p) \leq \begin{cases}
            \frac{p^2}{16}, & p \equiv 0 \mod 4,\\
            \frac{(p+2)(p-2)}{16}, & p \equiv 2 \mod 4.
        \end{cases}
    \end{equation}
    Inverting \eqref{eq:a(p)} yields
    \begin{equation*}
        p \geq \min\{4 \sqrt{a}, 2\sqrt{4a+1}\} = 4 \sqrt{a}.
    \end{equation*} 
    However, we know $p \equiv 0 \mod 2$, so we need to round up to the nearest even number. This gives the desired result.
\end{proof}

\section{Additional characterization of $F_\al$ and $F_\al^{-1}$}\label{piecewise}

\begin{claim}
\label{claim:left-cont}
    Suppose $F:\Z_{\geq 0} \to [0,1)$ is a nondecreasing function with $F(0) = 0$ and $F:\Z_{\geq 1} \to (0,1)$. Then $F^{-1}(x) := \min\{k \in \N : F(k) \geq x\}$ is left-continuous on $(0,1)$.
\end{claim}

Note that if the domain of $F$ is appropriately extended to $\R_{\geq 0}$, $F$ can be viewed as the cumulative distribution function of a discrete random variable, and this is a standard result. Nevertheless, we provide a short proof for completeness. 
    
\begin{proof}
    Fix $x_0 \in (0,1)$ and set $k_0 := F^{-1}(x_0) = \min\{k \in \N : F(k) \geq x_0\}$. By the minimality of $k_0$, $F(k_0-1) < x_0$, so we define $\delta := x_0-F(k_0-1) > 0$. Suppose $y$ satisfies $x_0-\delta < y < x_0$. Then $F(k_0-1) < y < x_0 \leq F(k_0)$. Since $F$ is nondecreasing, $F(i) \leq F(k_0-1)<y$ for all $0 \leq i \leq k_0-1$. But we also know $F(k_0) \geq y$, so the smallest natural number $j$ satisfying $F(j) \geq y$ is $j = k_0$. Hence $F^{-1}(y) = k_0$ for all $y \in (x_0-\delta, x_0)$. 
\end{proof}

\begin{lemma}[Piecewise formula for $F_4(k)$]
\label{lem:closedform}
For all $k, n \in \N$,
\begin{equation*}
    F_4(k) := \max_{1 \leq j \leq k} ~\frac{j}{j+\frac{\lceil 2\sqrt{2j-1}\rceil}{2}} = \begin{cases}
        \frac{s_0(n)}{s_0(n)+2n-1},& k \in [s_0(n), s_0(n)+\frac{n}{2})\\
        \frac{k}{k+2n-\frac{1}{2}},& k \in [s_0(n)+\frac{n}{2}, s_1(n))\\
        \frac{s_1(n)}{s_1(n)+2n-\frac{1}{2}},& k \in [s_1(n), s_1(n)+\frac{n}{2})\\
        \frac{k}{k+2n},& k \in [s_1(n)+\frac{n}{2}, s_2(n))\\
        \frac{s_2(n)}{s_2(n)+2n},&k \in [s_2(n), s_2(n)+\frac{n+1}{2})\\
        \frac{k}{k+2n+\frac{1}{2}},& k \in [s_2(n)+\frac{n+1}{2}, s_3(n))\\
        \frac{s_3(n)}{s_3(n)+2n+\frac{1}{2}},& k \in [s_3(n), s_3(n)+\frac{n+1}{2})\\
        \frac{k}{k+2n+1},& k \in [s_3(n)+\frac{n+1}{2}, s_0(n+1))
    \end{cases}
\end{equation*}
\end{lemma}

\begin{proof}
Recall that the integer sequences $s_i(n)$ are defined in Lemma \ref{lem:maxarea}. By Lemma \ref{lem:maxarea}, 
\begin{align}\label{eq:epsabcd}
    \sigma_4(s_0(n)) &= 4n, \nonumber\\
    \sigma_4(s_1(n)) &= \sigma_4(s_0(n)+1) = 4n+1, \nonumber\\
    \sigma_4(s_2(n)) &= \sigma_4(s_1(n)+1) = 4n+2,\\
    \sigma_4(s_3(n)) &= \sigma_4(s_2(n)+1) = 4n+3, \nonumber\\
    \sigma_4(s_0(n+1)) &= \sigma_4(s_3(n)+1) = 4n+4.\nonumber
\end{align}
Since $\sigma_4(k)$ is a nondecreasing function, this shows that $\sigma_4(k)$ is constant on the intervals $[s_0(n)+1,s_1(n)],~[s_1(n)+1,s_2(n)],~ [s_2(n)+1,s_3(n)],$ and $[s_3(n)+1,s_0(n+1)]$. This implies that $f_4(j) = \frac{j}{j+\frac{\sigma_4(j)}{2}-1}$ is a strictly increasing function on these intervals. To relate the values of $f_4$ between adjacent intervals, we compute the following identities for $n \geq 2$:
\begin{align}
\label{eq:fidentities}
    f_4(s_0(n)) &=f_4(s_0(n)+\delta_0(n)),~~\delta_0(n) := \frac{n}{2}-\frac{1}{4}+\frac{1}{4(2n-1)} \nonumber\\
    f_4(s_1(n)) &=f_4(s_1(n)+\delta_1(n)),~~\delta_1(n) := \frac{n}{2}-\frac{1}{8}-\frac{1}{8(4n-1)} \nonumber\\
    f_4(s_2(n)) &=f_4(s_2(n)+\delta_2(n)),~~\delta_2(n) := \frac{n}{2} \nonumber\\
    f_4(s_3(n)) &=f_4(s_3(n)+\delta_3(n)),~~\delta_3(n) := \frac{n}{2}+\frac{1}{8}-\frac{1}{8(4n+1)} \nonumber
\end{align}
The identities $f_4(s_i(n)) = f_4(s_i(n)+\delta_i(n))$, $i \in \{0,1,2,3\}$, may be easily verified using \eqref{eq:epsabcd}. Since $f_4(j)$ is a strictly increasing for $j \in [s_i(n)+1,s_{i+1}(n)]$, where $s_4(n):= s_0(n+1)$, we can use the above identities to precisely characterize the behavior of $F_4$. Observing that $\frac{n}{2} - \frac{1}{4} < \delta_i(n) < \frac{n}{2} + \frac{1}{4}$, it follows that $F_4(k) = f_4(s_j(n))$ for all $k \in [s_j(n), s_j(n)+ \frac{n}{2})$ and $F_4(k)= f_4(k)$ for all $k \in (s_j(n) +\frac{n}{2},s_{j+1}(n)]$. We compare the relative sizes of $f_4(s_i(n) + \frac{n}{2})$ and $f_4(s_i(n))$ by examining the exact value of $\delta_i(n)$ for each $i \in \{0,1,2,3\}$. Together, this information leads to the desired piecewise formula for all $n \geq 2$, which corresponds with $k \geq s_0(2) = 5$. We verify that the piecewise formula also holds for all $k \in [4]$ by separately checking the $n = 1$ case.  
\end{proof}

\begin{lemma}[Piecewise formula for $F_8(k)$]
\label{lem:closedform8}
For all $k, n \in \N$,
\begin{equation*}
    F_8(k) := \max_{1 \leq j \leq k} ~\frac{j}{j+\lceil 2\sqrt{j}\rceil + 1} = \begin{cases}
        \frac{t_0(n)}{t_0(n)+4n-1},& k \in [t_0(n), t_0(n)+n-1]\\
        \frac{k}{k+4n},& k \in [t_0(n)+n, t_1(n)-1]\\
        \frac{t_1(n)}{t_1(n)+4n},& k \in [t_1(n), t_1(n)+n-1]\\
        \frac{k}{k+4n+1},& k \in [t_1(n)+n, t_2(n)-1]\\
        \frac{t_2(n)}{t_2(n)+4n+1},&k \in [t_2(n), t_2(n)+n-1]\\
        \frac{k}{k+4n+2},& k \in [t_2(n)+n, t_3(n)-1]\\
        \frac{t_3(n)}{t_3(n)+4n+2},& k \in [t_3(n), t_3(n)+n-1]\\
        \frac{k}{k+4n+3},& k \in [t_3(n)+n, t_0(n+1)-1]
    \end{cases}
\end{equation*}
\end{lemma}

\begin{proof}
Recall that the integer sequences $t_i(n)$ are defined in Lemma \ref{lem:maxarea8}. By Lemma \ref{lem:maxarea8}, 
\begin{align}
\label{eq:sigma8vals}
    \sigma_8(t_0(n)) &= 8n, \nonumber\\
    \sigma_8(t_1(n)) &= \sigma_8(t_0(n)+1) = 8n+2, \nonumber\\
    \sigma_8(t_2(n)) &= \sigma_8(t_1(n)+1) = 8n+4,\\
    \sigma_8(t_3(n)) &= \sigma_8(t_2(n)+1) = 8n+6, \nonumber\\
    \sigma_8(t_0(n+1)) &= \sigma_8(t_3(n)+1) = 8n+8.\nonumber
\end{align}
Since $\sigma_8(k)$ is a nondecreasing function, this shows that $\sigma_8(k)$ is constant on the intervals $[t_0(n)+1,t_1(n)],~[t_1(n)+1,t_2(n)],~ [t_2(n)+1,t_3(n)],$ and $[t_3(n)+1,t_0(n+1)]$. This implies that $f_8(j) = \frac{j}{j+\frac{\sigma_8(j)}{2}-1}$ is a strictly increasing function on these intervals. To relate the values of $f_8$ between adjacent intervals, we compute the following identities for all integers $n \geq 1$:
\begin{align}
\label{eq:fidentities}
    f_8(t_0(n)) &=f_8(t_0(n)+\Delta_0(n)),~~\Delta_0(n) := n-\frac{3}{4}+\frac{1}{4(4n-1)} \nonumber\\
    f_8(t_1(n)) &=f_8(t_1(n)+\Delta_1(n)),~~\Delta_1(n) := n-\frac{1}{2} \nonumber\\
    f_8(t_2(n)) &=f_8(t_2(n)+\Delta_2(n)),~~\Delta_2(n) := n-\frac{1}{4} + \frac{1}{4(4n+1)}\nonumber\\
    f_8(t_3(n)) &=f_8(t_3(n)+\Delta_3(n)),~~\Delta_3(n) := n \nonumber
\end{align}

These identities $f_8(t_i(n)) = f_8(t_i(n)+\Delta_i(n))$, $i \in \{0,1,2,3\}$, may be easily verified using \eqref{eq:sigma8vals}. Since $f_8(j)$ is a strictly increasing for $j \in [t_i(n)+1,t_{i+1}(n)]$, where $t_4(n):= t_0(n+1)$, we can use the above identities to precisely characterize the behavior of $F_8$. In particular, it must be the case that $F_8(k) = f_8(t_i(n))$ for all $k \in [t_i(n), t_i(n)+ \lfloor \Delta_i(n) \rfloor]$ and $F_8(k) = f_8(k)$ for all $k \in [t_i(n)+ \lfloor \Delta_i(n) \rfloor + 1, t_{i+1}(n)]$. Simplifying $\lfloor \Delta_i(n) \rfloor$ affords the desired piecewise formula.  
\end{proof}

\begin{lemma} \label{lem:F4inv} If $d \in (0,\frac{1}{2}]$, $F_4^{-1}(d) = 1$. For all $d \in (\frac{1}{2},1)$,
\begin{equation*}
    F_4^{-1}(d) = \lceil h_4(d) \rceil, ~~~~ h_4(d) = \frac{d}{2(1-d)}\cdot \begin{cases}
        4n-1, & d \in \big(\frac{s_0(n)}{s_0(n)+2n-1}, \frac{s_1(n)}{s_1(n)+2n - 1/2}\big]\\
        4n, & d \in \big(\frac{s_1(n)}{s_1(n)+2n - 1/2}, \frac{s_2(n)}{s_2(n)+2n}\big]\\
        4n+1, & d \in \big(\frac{s_2(n)}{s_2(n)+2n}, \frac{s_3(n)}{s_3(n)+2n+1/2}\big]\\
        4n+2, & d \in \big(\frac{s_3(n)}{s_3(n)+2n+1/2}, \frac{s_0(n+1)}{s_0(n+1)+2n+1}\big]\\
    \end{cases}
\end{equation*}
\end{lemma}

\begin{lemma} \label{lem:F8inv} If $d \in (0,\frac{1}{4}]$, $F_8^{-1}(d) = 1$. For all $d \in (\frac{1}{4},1)$,
\begin{equation*}
    F_8^{-1}(d) = \lceil h_8(d) \rceil, ~~~~ h_8(d) = \frac{d}{1-d}\cdot \begin{cases}
        4n, & d \in \big(\frac{t_0(n)}{t_0(n)+4n-1}, \frac{t_1(n)}{t_1(n)+4n}\big]\\
        4n+1, & d \in \big(\frac{t_1(n)}{t_1(n)+4n}, \frac{t_2(n)}{t_2(n)+4n+1}\big]\\
        4n+2, & d \in \big(\frac{t_2(n)}{t_2(n)+4n+1}, \frac{t_3(n)}{t_3(n)+4n+2}\big]\\
        4n+3, & d \in \big(\frac{t_3(n)}{t_3(n)+4n+2}, \frac{t_0(n+1)}{t_0(n+1)+4n+3}\big]\\
    \end{cases}
\end{equation*}
\end{lemma}

\section{Additional constructions}

\begin{figure}[ht]
    \centering
    \includegraphics[width=0.71\textwidth]{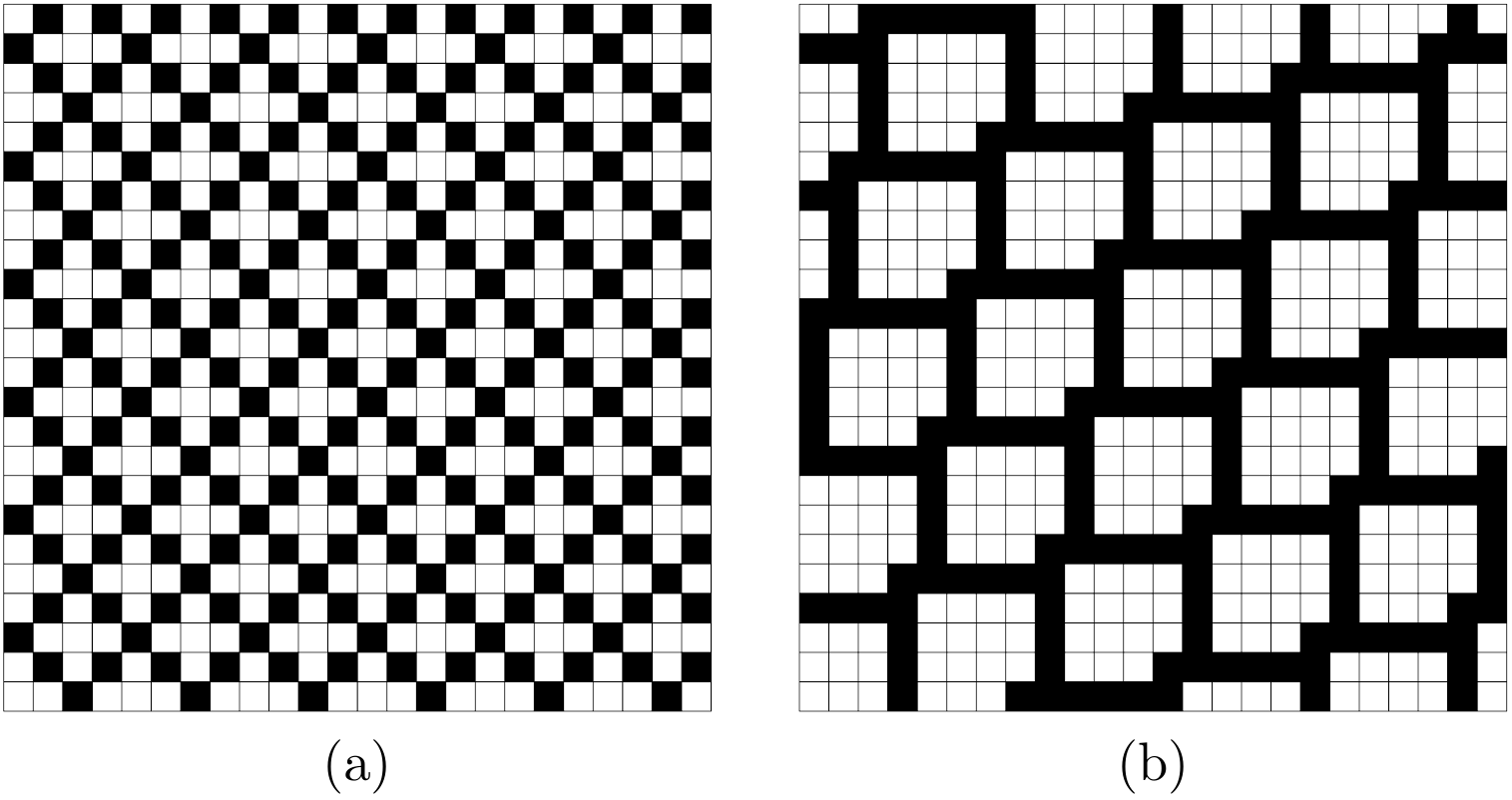}
    \caption{(a) and (b) are $24 \times 24$ pixel binary images (with white pixel density $\frac{5}{8}$) that illustrate tightness of all the Theorem \ref{thm:bounds} bounds at $d = \frac{5}{8}$. In the 8-connected case, the image example is a $4 \times 4$ grid of translated copies of (b).}
    \label{fig:5-15}
\end{figure}

\newpage

\begin{remark}
\label{rem:ABlattices}~\\
\begin{enumerate}
\item $\PP_4(\mc{A}_{k,n})$ tiles the plane with respect to the lattice generated by $\mathscr{A}_{k,n} := \begin{pmatrix}
k & -n \\
k & n
\end{pmatrix}$. \label{Atiling}
\item $\PP_4(\mc{B}_{k,n})$ tiles the plane with respect to the lattice generated by $\mathscr{B}_{k,n} := \begin{pmatrix}
k & -n \\
k+1 & n+1
\end{pmatrix}$. \label{Btiling}
\end{enumerate}
\end{remark}

\begin{figure}[ht]
\centering
\includegraphics[width=0.6\textwidth]{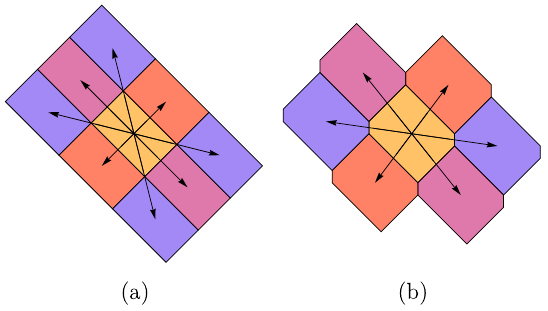}
\caption{(a) and (b) illustrate the surroundings of each tile in the Remark \eqref{rem:ABlattices} \eqref{Atiling} and \eqref{Btiling} lattice tilings, respectively. The corresponding translation vectors are also shown.}
\label{fig:tiling2}
\end{figure}

\end{document}